\documentclass[11pt]{amsart}

\usepackage{amssymb}
\usepackage{amsrefs}
\usepackage{overpic}
\usepackage{enumitem}   
\usepackage{graphicx} 
\usepackage{xcolor}
\usepackage{mathrsfs} 
\usepackage[a4paper,margin=1.5cm]{geometry}
\usepackage[hidelinks]{hyperref}
\graphicspath{{Figures/}}

\newtheorem{theorem}{Theorem}
\newtheorem{proposition}{Proposition} 
\newtheorem{corollary}{Corollary} 
\newtheorem{remark}{Remark}

\begin{document}
	
\title[On the cyclicity of hyperbolic polycycles]{On the cyclicity of hyperbolic polycycles}

\author[Claudio Buzzi, Armengol Gasull and Paulo Santana]
{Claudio Buzzi$^1$, Armengol Gasull$^2$ and Paulo Santana$^1$}

\address{$^1$ IBILCE--UNESP, CEP 15054--000, S. J. Rio Preto, S\~ao Paulo, Brazil}
\email{claudio.buzzi@unesp.br; paulo.santana@unesp.br}

\address{$^2$ Departament de Matem\`{a}tiques, Facultat de Ci\`{e}ncies, Universitat Aut\`{o}noma de Barcelona and Centre de Recerca Matem\`{a}tica, Spain}
\email{armengol.gasull@uab.cat}



\begin{abstract}
	Let $X$ be a planar smooth vector field with a polycycle $\Gamma^n$ with $n$ sides and all its corners,   that are at most $n$ singularities, being hyperbolic saddles. In this paper we study the cyclicity of $\Gamma^n$ in terms of the hyperbolicity ratios of these saddles,  giving explicit conditions that ensure that it is at least $k,$ for any $k\leqslant n.$ Our result extends old results and also provides a  more accurate proof of the known ones  because we rely on some recent powerful works that study in more detail the regularity with respect to initial conditions and parameters of the Dulac map of hyperbolic saddles for families of vector fields. We also prove that when $X$ is polynomial  there is  a polynomial perturbation (in general with degree much higher that the one of $X$) that attains each of the obtained lower bounds for the  cyclicities. Finally, we also study some related inverse problems and provide concrete examples of applications in the polynomial world.
\end{abstract}

\maketitle

\section{Introduction and Main Result}

Let $X$ be a planar smooth vector field (i.e. of class $C^\infty$). A \emph{graphic} $\Gamma$ for $X$ is a compact, non-empty invariant subset which is a continuous (but not necessarily homeomorphic) image of $\mathbb{S}^1$ and consists of a finite number of isolated singularities $\{p_1,\dots,p_n\}$ (not necessarily distinct) and a compatibly set of distinct regular orbits $\{L_1,\dots,L_n\}$ such that $p_i$ is the $\omega$-limit of $L_i$. A \emph{polycycle} is a graphic with a well defined first return map on one of its sides. A polycycle is \emph{hyperbolic} if all its singularities are hyperbolic saddles. Let $\Gamma^n$ denote a hyperbolic polycycle composed by the hyperbolic saddles $\{p_1,\dots,p_n\}$ (not necessarily distinct) and by the distinct regular orbits $\{L_1,\dots,L_n\},$ the sides of the polycycle, such that $p_i$ is the $\omega$-limit of $L_i,$ see Figure~\ref{Fig12}.

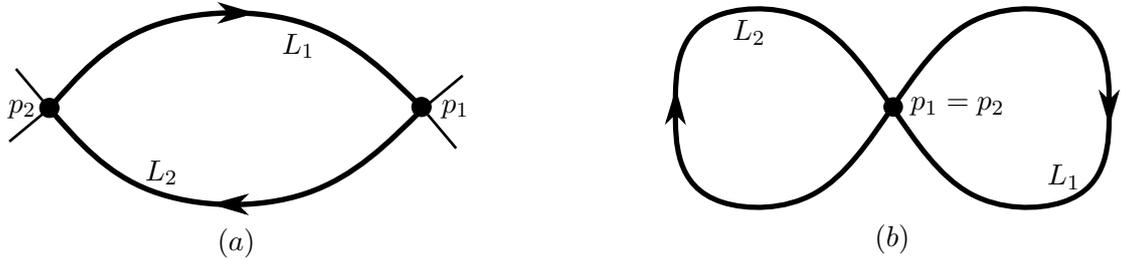
\begin{figure}[ht]
	\begin{center}
		\begin{minipage}{8.5cm}
			\begin{center} 
				\begin{overpic}[width=6cm]{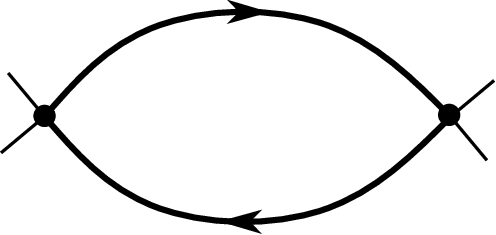} 
					\put(95,23){$p_1$}
					\put(0,23){$p_2$}
					\put(60,36){$L_1$}
					\put(30,8){$L_2$}
				\end{overpic}
				
				$(a)$
			\end{center}
		\end{minipage}
		\begin{minipage}{8.5cm}
			\begin{center} 
				\begin{overpic}[width=6cm]{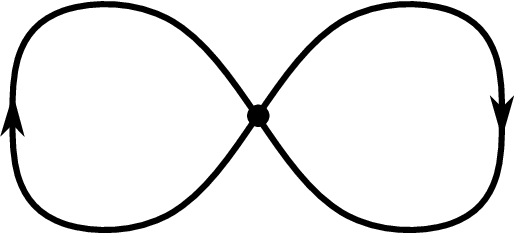} 
					\put(54,22.5){$p_1=p_2$}
					\put(84,6){$L_1$}
					\put(15,38){$L_2$}
				\end{overpic}
				
				$(b)$
			\end{center}
		\end{minipage}
	\end{center}
	\caption{Illustration of $\Gamma^2$, with $(a)$ distinct and $(b)$ non-distinct hyperbolic saddles.}\label{Fig12}
\end{figure}

Let $\lambda_i^s<0<\lambda_i^u$ be the associated eigenvalues of the saddle $p_i$, $i\in\{1,\dots,n\}$. The \emph{hyperbolicity ratio} of $p_i$ is the positive real number
\begin{equation}\label{3}
	r_i=\frac{|\lambda_i^s|}{\lambda_i^u}.
\end{equation}
The \emph{graphic number} of $\Gamma^n$ is the positive real number given by,
\begin{equation}\label{23}
	r(\Gamma^n)=\prod_{i=1}^{n}r_i.
\end{equation}
Cherkas \cite{Cherkas} proved that if $r(\Gamma^n)\neq1$, then $\Gamma^n$ has a well defined stability. More precisely, if $r(\Gamma^n)>1$, then $\Gamma^n$ is stable (i.e. it attracts the orbits in the region where the first return map is defined). Similarly, if $r(\Gamma^n)<1$, then $\Gamma^n$ is unstable. Since $r(\Gamma^n)$ depends continuously on smooth perturbations, it follows that if $r(\Gamma^n)\neq1$, then $\Gamma^n$ has no change of stability for small perturbations that do not break the polycycle. According with the terminology introduced by Sotomayor \cite[Section 2.2]{Soto}, when  $r(\Gamma^n)\neq1$ it is said that $\Gamma^n$ is \emph{simple}. 

Roughly speaking, we say that $\Gamma^n$ has \emph{cyclicity} greater or equal~$k$ inside a family of vector fields containing~$X$ if it is possible to bifurcate at least $k$ limit cycles from   $\Gamma^n$ for some arbitrarily small perturbations of $X$ inside this family  (a more rigorous definition shall be given latter). Several authors have results computing also exact cyclicities or upper bounds, but our results are restricted to give lower bounds. For instance, Andronov and Leontovich \cite{AndLeo1959} proved that if $n=1$ and $r(\Gamma^1)=r_1\neq1$, then the cyclicity of $\Gamma^1$ is at most one. For  accessible and didactic versions of this result, we refer to Andronov et al \cite[$\mathsection29$]{And1971} or Kuznetsov \cite[Section 6.2]{Kuz2004}. If $n=2$, Mourtada \cite{Mourtada2} proved that if $(r_1-1)(r_2-1)\ne0$, then the cyclicity of $\Gamma^2$ is at most $2$. Moreover, if $(r_1-1)(r_2-1)<0$, then it is~$2$ for suitable families. For $n\in\{3,4\}$, Mourtada \cites{Mourtada3,Mourtada4} also proved similar generic results, showing also the striking result that when $n=4$ there are generic families with cyclicity  $5.$ For more details, we refer to Roussarie \cite[Chapter $5$]{Roussarie}. We remark that to obtain such cyclicities, in general it is necessary to break the polycycles. To understand why this jump in the cyclicity happens when $n$ increases it is instructive to read the recent paper of Panazzolo~\cite{Panazzolo} where the author proposes a representative model for the breaking of hyperbolic polycycles.

Recently, Dukov~\cite{Duk2023} proved that for each $n\geqslant 2$, if $\Gamma^n$ satisfies again some generic conditions, then any limit cycle bifurcating from $\Gamma^n$ by a finite-dimensional perturbation has multiplicity at most $n$. 

On the other hand, on non-generic cases and with suitable perturbations,  it is known that the cyclicity can be much higher than $n.$  For instance, for $n=1$ (resp. $n=2$) Han and Zhu \cite{HanZhu2007} have provided an example of $\Gamma^1$ with cyclicity at least $5$ (resp. $12$), inside the polynomial systems of degree~$8$ (resp. $11$).  A higher cyclicity for $n=2$ is given by Tian and Han in \cite{TiaHan2017}. 

For a study of the cyclicity of \emph{persistent} polycycles (i.e. to obtain limit cycles without breaking the original polycycle), we refer to Marin and Villadelprat \cite{MarVil2022}. For other examples of lower bounds for the cyclicity of $\Gamma^n$ for low values of $n$, we refer to \cite{SheHanTia2020} and the references therein. We also refer to the works of Gasull et al \cite{GasManMan} and Han et al \cite{HanHuLiu2003} for the study of the stability of polycycles where the graphic number \eqref{23} is equal to $1$. 

In recent years there is an extension of some results, such as the one of Cherkas, to the case of planar non-smooth vector fields (also known as piecewise smooth or Filippov systems). See Santana \cite{San2023}.

Inspired by the work of Han et al \cite{HanWuBi2004}, in this paper we study under generic conditions the cyclicity of $\Gamma^n$, $n\geqslant 1$, both in the smooth and polynomial cases.

In a few words, the geometric idea behind the bifurcations of the limit cycles consists in breaking a given polycycle $\Gamma^n$ in ``sub-polycycles'' $\Gamma^{n-1}$, $\Gamma^{n-2},\dots$ by casting out its hyperbolic saddles one-by-one in such a way that at least one limit cycle bifurcates at each step, see Figure~\ref{Fig6}. 
\begin{figure}[h]		
		\begin{center}
			\begin{minipage}{5cm}
				\begin{center} 
					\begin{overpic}[width=4cm]{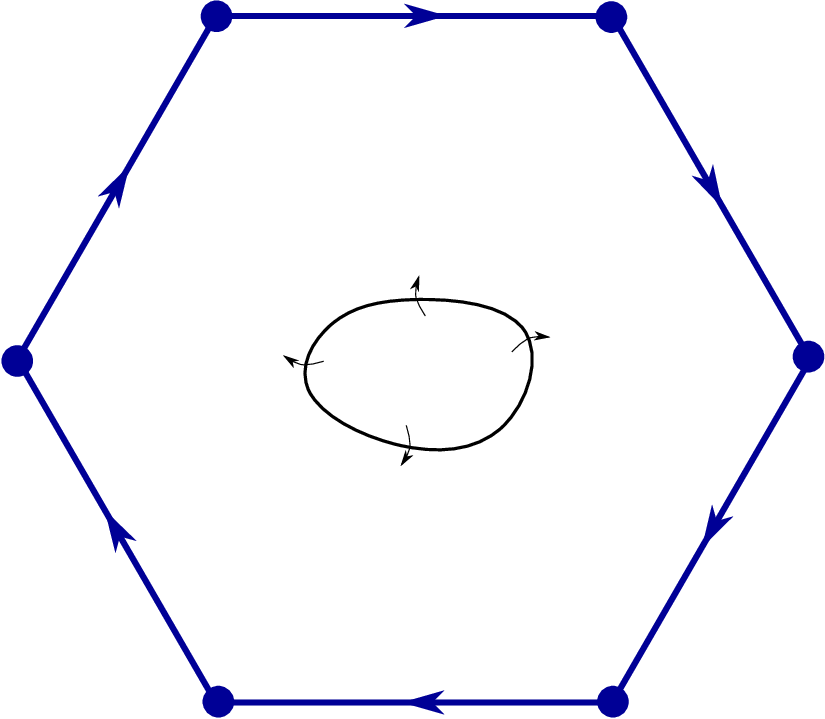} 
						\put(78,85){$p_1$}
						\put(14,85){$p_2$}
						\put(-10,43){$p_3$}
						\put(14,0){$p_4$}
						\put(78,0){$p_5$}	
						\put(101,43){$p_6$}
					\end{overpic}
				
					$\;$
				
					$(a)$ Unperturbed.
					
				\end{center}
			\end{minipage}
			\begin{minipage}{5cm}
				\begin{center} 
					\begin{overpic}[width=4cm]{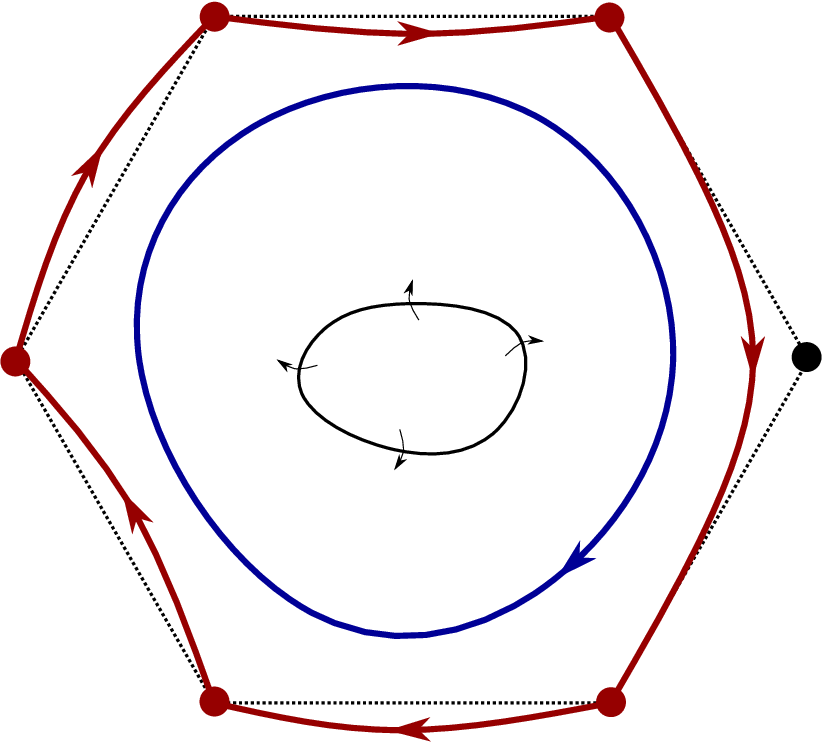} 
						\put(79,85){$p_1$}
						\put(12,85){$p_2$}
						\put(-10,43){$p_3$}
						\put(14,0){$p_4$}
						\put(78,0){$p_5$}	
						\put(101,43){$p_6$}
					\end{overpic}
				
					$\;$
				
					$(b)$ First perturbation.
					
				\end{center}
			\end{minipage}
			\begin{minipage}{5cm}
				\begin{center} 
					\begin{overpic}[width=4cm]{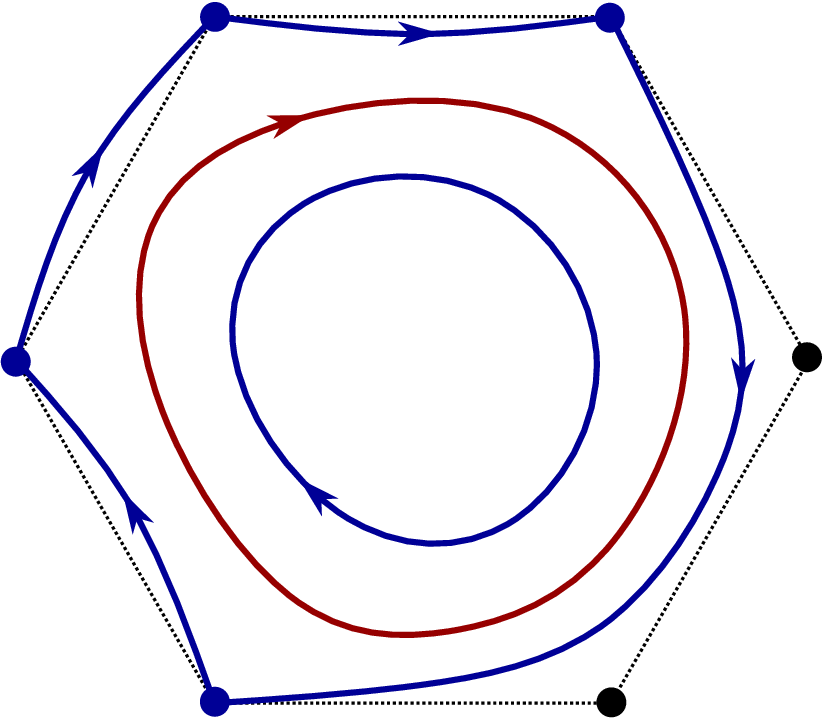} 
						\put(78,85){$p_1$}
						\put(14,85){$p_2$}
						\put(-10,43){$p_3$}
						\put(14,0){$p_4$}
						\put(78,0){$p_5$}	
						\put(101,43){$p_6$}
					\end{overpic}
				
					$\;$
					
					$(c)$ Second perturbation.
					
				\end{center}
			\end{minipage}
		\end{center}
	\caption{Illustration of the bifurcation process. Blue means stable and red means unstable. Colors available in the online version.}\label{Fig6}
\end{figure} 
More precisely, if at a given step the polycycles $\Gamma^{n_0}$ and $\Gamma^{n_0-1}$ have opposite stabilities, then it follows from the Poincar\'e-Bendixson Theorem (and some technical results) that at least one limit cycle of odd multiplicity bifurcates when we break from $\Gamma^{n_0}$ to $\Gamma^{n_0-1}$. To ensure these opposite stabilities it is sufficient to have $(r(\Gamma^{n_0})-1)(r(\Gamma^{n_0-1})-1)<0$, see~\eqref{23}. Moreover when casting out the hyperbolic saddles we do not need to follow the ``canonical'' indexation $\{p_1,\dots,p_6\}$, as in Figure~\ref{Fig6}. Rather at each step we can choose which singularity to expel in order to maximize the number of stability changes and thus the number of limit cycles. At Figure~\ref{Fig6} for example, one could start the process by expelling $p_4$ and then $p_1$ and etc.

As we shall see in Proposition~\ref{Main2}, any possible combination of $n$ hyperbolicity ratios \eqref{3} $r_1,\dots,r_n$ is realizable by a polynomial vector field of degree $n$. Hence, the possibility to choose the most convenient singularity to expel at each steep of the bifurcation process is very important in order to obtain better lower bounds for the number of limit cycles.

Therefore the main objective of this paper is to formalize this geometric idea, which includes deve\-loping the technical machinery necessary to it. At the end we also apply these ideas on concrete polynomial vector fields. In particular, for instance we prove that Figure~\ref{Fig6} is realizable by a polynomial vector field of degree six, see Proposition~\ref{Main3}.

To state precisely our main result we need to introduce some notations. Given $r\geqslant1$ \emph{finite}, let $C^r(\mathbb{R}^2,\mathbb{R}^2)$ be the set of the functions $f\colon\mathbb{R}^2\to\mathbb{R}^2$ of class $C^r$. Given $f\in C^r(\mathbb{R}^2,\mathbb{R}^2)$, a compact set $B\subset\mathbb{R}^2$ and $\varepsilon>0$, let $V(f,B,\varepsilon)\subset C^r(\mathbb{R}^2,\mathbb{R}^2)$ be the set of $C^r$-functions $g\colon\mathbb{R}^2\to\mathbb{R}^2$ such that  
	\[\max_{\substack{x\in B \\ |k|\leqslant r}}\left|\left|\frac{\partial^{|k|}f}{\partial x_1^{k_1}\partial x_2^{k_2}}(x)-\frac{\partial^{|k|}g}{\partial x_1^{k_1}\partial x_2^{k_2}}(x)\right|\right|<\varepsilon,\]
where $x=(x_1,x_2)$, $k=(k_1,k_2)\in\mathbb{Z}^2_{\geqslant0}$ and $|k|=k_1+k_2$. The Whitney's weak $C^r$-topology~\cite[Chapter 2]{Hirsch} is the topology on $C^r(\mathbb{R}^2,\mathbb{R}^2)$ having the family of all such $V(f,B,\varepsilon)$ as a sub-base. In other words, it is the smaller topology that contains all such $V(f,B,\varepsilon)$. Let now $\mathcal{P}$ be the set of planar polynomial vector fields \emph{of any degree}. Since $\mathcal{P}\subset C^r(\mathbb{R}^2,\mathbb{R}^2)$, we can endow $\mathcal{P}$ with the \emph{subspace topology} $\tau_r$, inherited from Whitney's weak $C^r$-topology. Hence, we set the topological space $\mathcal{P}^r=(\mathcal{P},\tau_r)$. Observe that two vector fields $X$, $Y\in\mathcal{P}^r$ are close if there is a ``big'' compact $B\subset\mathbb{R}^2$ and a small $\varepsilon>0$ such that 
	\[\max_{\substack{x\in B \\ |k|\leqslant r}}\left|\left|\frac{\partial^{|k|}X}{\partial x_1^{k_1}\partial x_2^{k_2}}(x)-\frac{\partial^{|k|}Y}{\partial x_1^{k_1}\partial x_2^{k_2}}(x)\right|\right|<\varepsilon.\]
Let $\mathfrak{X}=C^\infty(\mathbb{R}^2,\mathbb{R}^2)$ be the set of planar smooth vector fields. We endow $\mathfrak{X}$ with Whitney's strong $C^\infty$-topology $\tau_\infty$, see \cite[Chapter 2]{Hirsch} and \cite[Section~$2.3$]{GolGui1973}. Let $\mathfrak{X}^\infty=(\mathfrak{X},\tau_\infty)$. Roughly speaking, $Y_n\to X$ in $\mathfrak{X}^\infty$ if and only if for every $r\geqslant0$ finite there is a compact $B_r\subset\mathbb{R}^2$ and $n_r\in\mathbb{N}$ such that
	\[\lim\limits_{n\to\infty}\max_{\substack{x\in B_r \\ |k|\leqslant r}}\left|\left|\frac{\partial^{|k|}Y_n}{\partial x_1^{k_1}\partial x_2^{k_2}}(x)-\frac{\partial^{|k|}X}{\partial x_1^{k_1}\partial x_2^{k_2}}(x)\right|\right|=0\]
and $\left.Y_n\right|_{\mathbb{R}^2\setminus B_r}=\left.X\right|_{\mathbb{R}^2\setminus B_R}$, for every $n\geqslant n_r$, see Golubitsky and Guillemin~\cite[p. $43$]{GolGui1973}. Given $X\in\mathfrak{X}^\infty$ an example of convergence $Y_n\to X$, in $\mathfrak{X}^\infty$, that we shall use in this paper is the one given by $Y_n=X+\frac{1}{n}\Phi$, where $\Phi\in C^\infty(\mathbb{R}^2,\mathbb{R}^2)$ has compact support.
	
When interested only in a particular compact set $B$, we may restrict $\mathfrak{X}^\infty$ to it and say that $Y_n\to X$ in $\mathfrak{X}^\infty$ restricted to $B$ if for every $r\geqslant0$ finite we have,
	\[\lim\limits_{n\to\infty}\max_{\substack{x\in B \\ |k|\leqslant r}}\left|\left|\frac{\partial^{|k|}Y_n}{\partial x_1^{k_1}\partial x_2^{k_2}}(x)-\frac{\partial^{|k|}X}{\partial x_1^{k_1}\partial x_2^{k_2}}(x)\right|\right|=0.\]
Let $\Gamma^n$ be a hyperbolic polycycle composed by the (not necessarily distinct) hyperbolic saddles $\{p_1,\dots,p_n\}$. Let $I_n$ be the set of the permutations of $n$ elements. Given $\sigma\in I_n$ let
	\[R_{0,\sigma}=R_{1,\sigma}^{-1}, \quad R_{i,\sigma}=\prod_{j=1}^{i}r_{\sigma(j)}, \quad i\in\{1,\dots,n\},\]
where $r_k$ is the hyperbolicity ratio \eqref{3} of $p_k$. Let also
	\[\Delta(\Gamma^n,\sigma)=\#\{i\in\{1,\dots,n\}\colon (R_{i,\sigma}-1)(R_{i-1,\sigma}-1)<0\},\]
where $\#I$ denotes the cardinality of $I$. Finally, let
	\[\Delta(\Gamma^n)=\max\{\Delta(\Gamma^n,\sigma)\colon\sigma\in I_n\}.\]	
Observe that $0\leqslant\Delta(\Gamma^n)\leqslant n$. In particular, $\Delta(\Gamma^n)=0$ if, and only if, $r_1=\dots=r_n=1$. 

Inspired by Roussarie~\cite[Definition $12$]{Roussarie}, we now properly define  what we mean when we say that  a polycycle~$\Gamma^n$ of a vector field $X,$ that belongs to a  topological spaces $\mathfrak{X},$  has \emph{cyclicity greater than or equal~$k.$} Given two compacts $C_1$, $C_2\subset\mathbb{R}^2$, recall that the \emph{Hausdorff distance} between them is given by,
	\[d_H(C_1,C_2)=\max\Big\{\sup_{x\in C_1}d(x,C_2),\sup_{x\in C_2}d(C_1,x)\Big\},\]
where $d(x,C)=\inf\{||x-y||\colon y\in C\}$ is the usual distance between a point and a set in the euclidean space.	Then, we will say that $\textit{Cycl }(X,\mathcal{X},\Gamma^n)\geqslant k$  if, given any $\varepsilon>0$ there exists a vector field $Y_\varepsilon\in\mathcal{X},$ such that  it has at least $k$ limit cycles $\gamma_j(\varepsilon),j=1,\ldots,k$ such that
	\[\max\left\{d_H(\gamma_j(\varepsilon),\Gamma^n)\colon j\in\{1,\dots,k\}\right\}<\varepsilon\]		
and $Y_\varepsilon$ tends to $X$ when $\varepsilon$ goes to $0$.
	
Our main result is the following. 

\begin{theorem}\label{Main1}
	Let $\mathcal{X}$ be one of the topological spaces $\mathfrak{X}^\infty$ or $\mathcal{P}^r$, for some $r\geqslant1$. If $X\in\mathcal{X}$ has a hyperbolic polycycle $\Gamma^n$, then $\textit{Cycl }(X,\mathcal{X},\Gamma^n)\geqslant\Delta(\Gamma^n)$.
\end{theorem}

Given a polycycle $\Gamma^n$, the \emph{trivial permutation} $\tau\in I_n$ of $\Gamma^n$ is the permutation of the indexes of $\{p_1,\dots,p_n\}$ such that $p_{i+1}$ and $p_i$ are the $\alpha$ and $\omega$-limits of $L_i$, respectively, with $p_{n+1}=p_1$. See Figure~\ref{Fig13}.
\begin{figure}[h]
	\begin{center}
		\begin{minipage}{7cm}
			\begin{center} 
				\begin{overpic}[height=4cm]{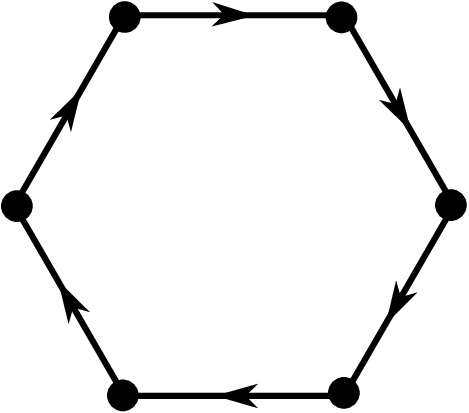} 
					\put(101,43){$p_6$}
					\put(78,85){$p_1$}
					\put(14,85){$p_2$}
					\put(-9,43){$p_3$}
					\put(14,0){$p_4$}
					\put(78,0){$p_5$}
					
					\put(50,76){$L_1$}
					\put(18,63){$L_2$}
					\put(12,32){$L_3$}
					\put(39,9){$L_4$}
					\put(72,23){$L_5$}
					\put(74,59){$L_6$}
					
				\end{overpic}
				
				$(a)$
			\end{center}
		\end{minipage}
		\begin{minipage}{7cm}
			\begin{center} 
				\begin{overpic}[height=4.1cm]{Fig26.eps}
					\put(101,43){$p_6$}
					\put(78,85){$p_3$}
					\put(14,85){$p_2$}
					\put(-9,43){$p_4$}
					\put(14,0){$p_1$}
					\put(78,0){$p_5$}
					
					\put(50,76){$L_3$}
					\put(18,63){$L_2$}
					\put(12,32){$L_4$}
					\put(39,9){$L_1$}
					\put(72,23){$L_5$}
					\put(74,59){$L_6$}
				\end{overpic}
				
				$(b)$
			\end{center}
		\end{minipage}
	\end{center}
	\caption{Illustration of $\Gamma^6$ with $(a)$ trivial and $(b)$ non trivial permutation on the indexes of the singularities.}\label{Fig13}
\end{figure} 

The case $X\in\mathfrak{X}^\infty$ of our main result is an extension of~\cite[Theorem~$1.1$]{HanWuBi2004} that corresponds to the case $\Delta(\Gamma^n,\tau)=n$  where $\tau$ is the trivial permutation of~$\Gamma^n$. Our proof is inspired in the ideas of that  paper but it is more detailed and transparent because it relies on recent results of Marin and Villadelaprat \cite{MarVil2020,MarVil2021,MarVil2024} that give regularity properties with respect initial conditions and parameters of the Dulac map associated to the hyperbolic sectors of hyperbolic saddles for families of vector fields, see Section~\ref{Sec2.1}. We comment in more detail about the differences between our proof and that of \cite{HanWuBi2004} in Remark~\ref{Remark7}. 

The result in the polynomial case $X\in\mathcal{P}^r$ is totally new. It is motivated by the only open problem left in order to get a complete characterization of the structurable stable polynomials vector fields of degree $n$ with the topology of the coefficients. This open problem consists on knowing whether non-hyperbolic limit cycles of odd multiplicity can be structurable stables or not, see~\cite{San2023,Sha1987,Soto1985}. Although we have not advanced on  this question  we have achieved a related result. More concretely, under generic conditions, we have been able to bifurcate $n$ limit cycles from a polycycle $\Gamma^n$ of a polynomial vector field $X$ with a polynomial perturbation and without losing control of its derivatives in any prescribed compact. Unfortunately, we have not been able to perform this unfold with the degree of the perturbations equals to the one of $X$.

In a few words, in the case $\Delta(\Gamma^n)=n$ what we prove is that starting from $\Gamma^n$ we can perturb $X$ such that from $\Gamma^n$ bifurcate a limit cycle and a new polycycle $\Gamma^{n-1} ,$ satisfying similar hypotheses to the ones of $\Gamma^n.$ Then, this process can be repeated $n-1$ times until arriving to $n$ limit cycles, all near~$\Gamma^n.$ The technical part of the proof is the control of the continuity and differentiability, with respect to initial conditions and parameters,  of the various return maps associated to the appearing polycycles.

The paper is structured as follows. In Section~\ref{Sec2} we recall some preliminary results about the Dulac and the displacement maps. In Section~\ref{Sec3} we work on the displacement map between non-subsequent saddles. In Section~\ref{Sec4} we prove some technical lemmas about the approximation of planar smooth functions by polynomials functions; the existence of positive or negative invariant regions; and the perturbation of periodic orbits. Theorem~\ref{Main1} is proved in Section~\ref{Sec5}. In Section~\ref{Sec6} we solve the inverse problem of constructing a vector field $X$ with a polycycle $\Gamma^n$ with any given set of prescribed hyperbolicity ratios and we apply our techniques to a concrete polynomial example. At Section~\ref{Sec7} we include some brief considerations about the current state of the theory of unfolding of polycycles and how our results can be applied to it, specially in the polynomial case.

\section{Dulac and displacement maps}\label{Sec2}

\subsection{The Dulac map}\label{Sec2.1}

Let $X_\mu$ be a planar smooth vector field depending on a smooth way on a parameter $\mu\in\mathbb{R}^N$, $N\geqslant1$, defined in a neighborhood of a hyperbolic saddle $p$ at $\mu=0$. Since $p$ is hyperbolic, it follows that if $\Lambda\subset\mathbb{R}^N$ is a small enough neighborhood of the origin, then the perturbation $p(\mu)$ of $p$ is well defined and it is also a hyperbolic saddle, for every $\mu\in\Lambda$. Moreover, restricting $\Lambda$ if necessary, it is well known (see Lemma~$4.3$ of \cite{MarVil2021}) that there are neighborhoods $U\subset\mathbb{R}^2$ of $p$ and $V\subset\mathbb{R}^2$ of the origin $\mathcal{O}$, and a smooth map $\Phi:U\times\Lambda\to V$ such that for each $\mu\in\Lambda$ the map $\Phi(\cdot,\mu)\colon U\to V$ is a smooth change of coordinates that sends $p(\mu)$ to the origin $\mathcal{O}$ and its (local) unstable and stable manifolds $\ell^u(\mu)$ and $\ell^s(\mu)$ to the axis $Ox$ and $Oy$, respectively. By abuse of notation we still denote this new vector field by $X_\mu$. Given $\varepsilon>0$ small, let $\sigma\colon(-\varepsilon,\varepsilon)\times\Lambda\to\Sigma_\sigma$ and $\tau\colon(-\varepsilon,\varepsilon)\times\Lambda\to\Sigma_\tau$ be two $C^\infty$ transverse sections to $X_\mu$ defined by
	\[\sigma(s,\mu)=(\sigma_1(s,\mu),\sigma_2(s,\mu)), \quad \tau(s,\mu)=(\tau_1(s,\mu),\tau_2(s,\mu)),\]
and such that $\sigma_1(0,\mu)=0$ and $\tau_2(0,\mu)=0$, for every $\mu\in\Lambda$. Suppose also that if $s>0$, then $\sigma_1(s,\mu)>0$ and $\tau_2(s,\mu)>0$. Let 
	\[\Sigma_\sigma^+=\{\sigma(s,\mu)\in\Sigma_\sigma\colon s>0\}, \quad \Sigma_\tau^+=\{\tau(s,\mu)\in\Sigma_\sigma\colon s>0\}.\]
Let also $\varphi(t,x;\mu)$ be the solution of $X_\mu$ passing through $x\in V$ at $t=0$. On the first quadrant, $\varphi$ defines a transition map $\Sigma_\sigma^+\mapsto\Sigma_\tau^+$, which can be seen as a map
\begin{equation}\label{53}
	D\colon(0,\varepsilon)\times\Lambda\to(0,\varepsilon),
\end{equation}
due to the transverses section $\sigma$ and $\tau$. See Figure~\ref{Fig16}. 
\begin{figure}[ht]
	\begin{center}
		\begin{overpic}[height=5cm]{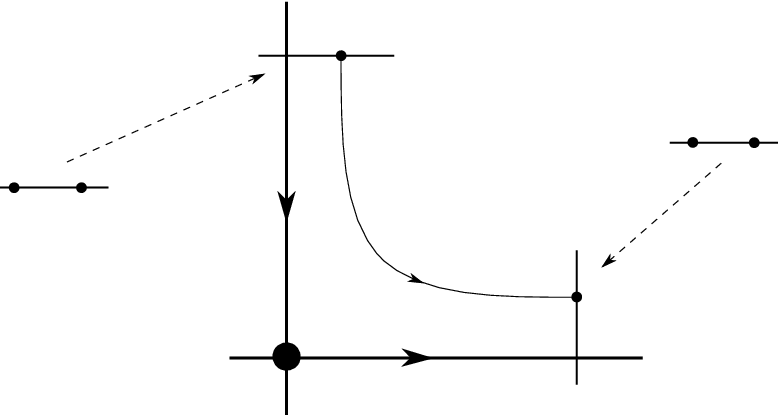} 
			\put(52,45){$\Sigma_\sigma$}
			\put(23,36){$\sigma$}
			\put(1,25){$0$}
			\put(9,26){$s$}
			
			\put(72,23){$\Sigma_\tau$}
			\put(87,24){$\tau$}
			\put(87.5,37){$0$}
			\put(94,37){$D(s)$}
			
			\put(32,2){$\mathcal{O}$}
			\put(28,50){$Oy^+$}
			\put(80,3){$Ox^+$}
			\put(48,25){$\varphi(\cdot,\sigma(s);\mu)$}
		\end{overpic}
	\end{center}
	\caption{The Dulac map near a hyperbolic saddle.}\label{Fig16}
\end{figure}
The map \eqref{53} is the \emph{Dulac map}. In recently years, Marin and Villadelprat \cites{MarVil2020,MarVil2021,MarVil2024} provided a full characterization of the Dulac map (and also of the Dulac time). In what follows, we will state some  properties of the Dulac map, that will be sufficient for the objectives of this paper. We recall that $r(\mu)$ denotes the hiperbolicity ratio \eqref{3} of the hyperbolic saddle $p(\mu)$, $\mu\in\Lambda$.

\begin{proposition}\label{MV1}
	If $r(0)>1$, then the Dulac map \eqref{53} can be extended to $s=0$ in a $C^1$-way and can be written as
		\[D(s,\mu)=\Delta_{00}(\mu)s^{r(\mu)}+\mathscr{R}(s,\mu),\]
	where $\Delta_{00}\colon\Lambda\to\mathbb{R}$ is a strictly positive function of class $C^\infty$ and $\mathscr{R}\colon[0,\varepsilon)\times\Lambda\to\mathbb{R}$ is a function of class $C^1$ satisfying
		\[\mathscr{R}(0,\mu)=0, \quad \frac{\partial\mathscr{R}}{\partial s}(0,\mu)=0, \quad \frac{\partial\mathscr{R}}{\partial\mu_j}(0,\mu)=0,\]
	for every $\mu\in\Lambda$ and $j\in\{1,\dots,N\}$. In particular,
		\[\lim\limits_{s\to0^+}D(s,\mu)=0, \quad \lim\limits_{s\to0^+}\frac{\partial D}{\partial s}(s,\mu)=0, \quad \lim\limits_{s\to0^+}\frac{\partial D}{\partial\mu_j}(s,\mu)=0,\]
	for every $\mu\in\Lambda$ and $j\in\{1,\dots,N\}$.
\end{proposition}

\begin{remark}
	Under the notation of \cite[Theorem $B$]{MarVil2021}, Proposition~\ref{MV1} follows by taking $L=r(0)$ and by replacing the flat term by its $C^1$ extension, which existence follows from $r(0)>1$. See Remark~$3$ right after \cite[Theorem $B$]{MarVil2021}.
\end{remark}

Observe that if we invert the time variable, then the solution of $-X_\mu$ defines a transition map $\Sigma_{\tau}^+\mapsto\Sigma_\sigma^+$ which can be seen as a map
\begin{equation}\label{54}
	D^{-1}\colon(0,\varepsilon)\times\Lambda\to(0,\varepsilon),
\end{equation}
satisfying
	\[D^{-1}(D(s),\mu)=s, \quad D(D^{-1}(s),\mu)=s,\]
for every $\mu\in\Lambda$. Since the hyperbolicity ratio of $\mathcal{O}$ in relation to $-X_\mu$ is given by $r(\mu)^{-1}$ and $D^{-1}$ is the Dulac map of $-X_\mu$, we also have the following result.

\begin{proposition}\label{MV2}
	If $r(0)<1$, then the Dulac map \eqref{54} can be extended to $s=0$ in a $C^1$-way and can be written as
		\[D^{-1}(s,\mu)=\Delta_{00}^*(\mu)s^{\frac{1}{r(\mu)}}+\mathscr{R}^*(s,\mu),\]
	where $\Delta_{00}^*\colon\Lambda\to\mathbb{R}$ is a strictly positive function of class $C^\infty$ and $\mathscr{R}^*\colon[0,\varepsilon)\times\Lambda\to\mathscr{R}$ is a function of class $C^1$ satisfying
		\[\mathscr{R}^*(0,\mu)=0, \quad \frac{\partial\mathscr{R}^*}{\partial s}(0,\mu)=0, \quad \frac{\partial\mathscr{R}^*}{\partial\mu_j}(0,\mu)=0,\]
	for every $\mu\in\Lambda$ and $j\in\{1,\dots,N\}$. In particular,
		\[\lim\limits_{s\to0^+}D^{-1}(s,\mu)=0, \quad \lim\limits_{s\to0^+}\frac{\partial D^{-1}}{\partial s}(s,\mu)=0, \quad \lim\limits_{s\to0^+}\frac{\partial D^{-1}}{\partial\mu_j}(s,\mu)=0,\]
	for every $\mu\in\Lambda$ and $j\in\{1,\dots,N\}$.
\end{proposition}

\begin{remark}[Theorem $B$ and Lemma $4.3$ of \cite{MarVil2021}]\label{Remark1}
	The smooth change of coordinates $\Phi\colon U\times\Lambda\to V$ that sends $p(\mu)$ to the origin and the unstable and stable manifolds $\ell^u(\mu)$ and $\ell^s(\mu)$ to the axis $Ox$ and $Oy$ is not necessarily for the characterization of the Dulac map. Moreover, the transverse sections $\Sigma_\sigma$ and $\Sigma_\tau$ does not need to be sufficiently close to $p(\mu)$. In particular, no normal form is needed to apply Propositions~\ref{MV1} and \ref{MV2}.
\end{remark}

\subsection{The displacement map}\label{Sec2.2}

Let $X$ be a planar smooth vector field with polycycle $\Gamma^n$ composed by $n$ hyperbolic saddles $p_i$, $i\in\{1,\dots,n\}$. In what follows, $\Gamma^n$ is endowed with the trivial permutation unless explicitly stated otherwise. Without loss of generality, suppose that $\Gamma^n$ is oriented in the clockwise way. Let $L_i$ be the regular orbit of $X$ from $p_{i+1}$ to $p_i$ (i.e. $\omega(L_i)=p_i$ and $\alpha(L_i)=p_{i+1}$, with $p_{n+1}=p_1$). Let $X=(P,Q)$ and define $X^\perp=(-Q,P)$. For each $i\in\{1,\dots,n\}$, let $x_i\in L_i$ and let $l_i$ be the normal section of $L_i$, at $x_i$, with the directed vector
\begin{equation}\label{55}
	v_i=\frac{1}{||X^\perp(x_i)||}X^\perp(x_i).
\end{equation}
Let $X_\mu$ be a smooth perturbation of $X$, with $\mu\in\mathbb{R}^N$, $N\geqslant 1$, and $X_0=X$. Since the saddles $p_i$ are hyperbolic, it follows that if $\Lambda\subset\mathbb{R}^N$ is a small enough neighborhood of the origin and $\mu\in\Lambda$, then the perturbation $p_i(\mu)$ of $p_i$ is well defined and it is also a hyperbolic saddle. For each $i\in\{1,\dots,n\}$, let $L_i^s(\mu)$ and $L_i^u(\mu)$ be the perturbations of $L_i$ such that $\omega(L_i^s(\mu))=p_i(\mu)$ and $\alpha(L_i^u(\mu))=p_{i+1}(\mu)$. Let also $x_i^s(\mu)$ and $x_i^u(\mu)$ be the intersections of $L_i^s(\mu)$ and $L_i^u(\mu)$ with $l_i$, respectively. See Figure~\ref{Fig1}. 
\begin{figure}[ht]
	\begin{center}
		\begin{minipage}{8.5cm}
			\begin{center} 
				\begin{overpic}[width=8cm]{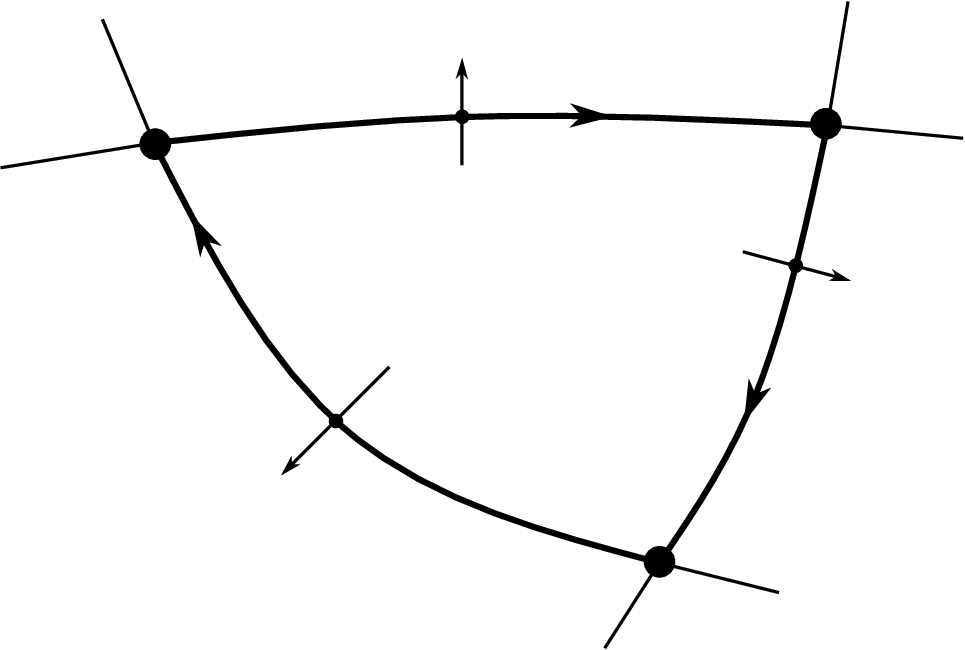} 
					\put(88,57){$p_2$}
					\put(17,55){$p_3$}
					\put(68,4){$p_1$}
					\put(49,51.5){$x_2$}
					\put(37,23){$x_3$}
					\put(76,37){$x_1$}
					\put(70,57){$L_2$}
					\put(20,32){$L_3$}
					\put(76,18){$L_1$}
					\put(49,60){$l_2$}
					\put(31,15){$l_3$}
					\put(87,40){$l_1$}
				\end{overpic}
				
				Unperturbed
			\end{center}
		\end{minipage}
		\begin{minipage}{8.5cm}
			\begin{center} 
				\begin{overpic}[width=8cm]{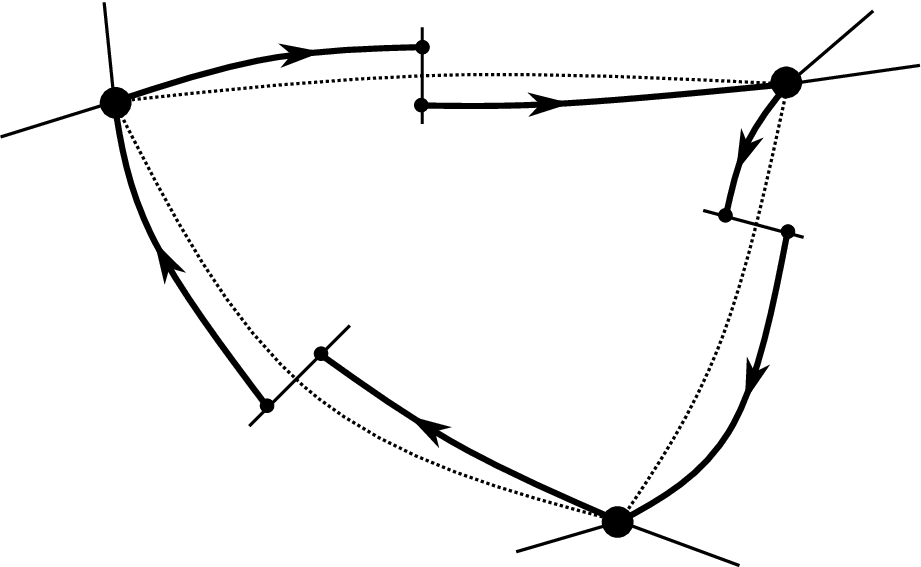} 
					\put(84,57){$p_2$}
					\put(6,53){$p_3$}
					\put(65,0){$p_1$}
					
					\put(47,56){$x_2^u$}
					\put(39,49){$x_2^s$}
					\put(25,59){$L_2^u$}
					\put(55,44){$L_2^s$}
					
					\put(29,25){$x_3^u$}
					\put(29,14){$x_3^s$}
					\put(49,16){$L_3^u$}
					\put(12,28){$L_3^s$}
					
					\put(75,33){$x_1^u$}
					\put(85,38){$x_1^s$}
					\put(73,45){$L_1^u$}
					\put(84,18){$L_1^s$}
				\end{overpic}
				
				Perturbed
			\end{center}
		\end{minipage}
	\end{center}
	\caption{An example of a perturbation of $\Gamma^3$, with $d_1(\mu)<0$, $d_2(\mu)>0$ and $d_3(\mu)<0$. For simplicity, we omitted the dependence on $\mu$ in the expressions of $x_i^{s,u}$ and $L_i^{s,u}$.}\label{Fig1}
\end{figure}
Observe that a point $q\in l_i$ can be represented as $q=x_i+\lambda v_i$. Hence, for each $i\in\{1,\dots,n\}$, let $b_i^s(\mu)$, $b_i^u(\mu)\in\mathbb{R}$ be such that,
\begin{equation}\label{17}
	x_i^s(\mu)=x_i+b_i^s(\mu)v_i, \quad x_i^u(\mu)=x_i+b_i^u(\mu)v_i.
\end{equation}
The \emph{displacement function} $d_i\colon\Lambda\to\mathbb{R}$ is given by,
\begin{equation}\label{4}
	d_i(\mu)=b_i^u(\mu)-b_i^s(\mu).
\end{equation}
It follows from Perko \cite[Lemma $2$]{Perko1994} and Guckenheimer and Holmes \cite[Section 4.5]{GuckHolmes} that if $\Lambda$ is a small enough, then $d_i$ is a well defined function of class $C^\infty$. Moreover if we write $X_\mu(x)=X(x)+K(x,\mu)$, with $K(x,0)\equiv0$, then the partial derivatives of $d_i$ at $\mu=0$ are given by
\begin{equation}\label{5}
	\frac{\partial d_i}{\partial \mu_j}(0)=\frac{1}{||X(x_i)||}\int_{-\infty}^{+\infty}e^{-\int_{0}^{t}\operatorname{div} X(\gamma_i(s))\;ds}X(\gamma_i(t))\land\frac{\partial K}{\partial\mu_j}(\gamma_i(t),0)\;dt,
\end{equation}
where $(x_1,y_1)\land(x_2,y_2)=x_1y_2-x_2y_1$ and $\gamma_i(t)$ is the parametrization of $L_i$ given by the solution of $X$, with initial condition $\gamma_i(0)=x_i$. For more details, we also refer to the survey of Blows and Perko~\cite{Blows}.

\begin{remark}\label{Remark2}
	From~\cite[Lemma $1$]{Perko1994} and~\cite[Section 4.5]{GuckHolmes} it can be seen that the functions $b_i^u$, $b_i^s\colon\Lambda\to\mathbb{R}$, given at \eqref{4} are also of class $C^\infty$, individually.
\end{remark} 

\begin{remark}
	We observe that both in \cite[Lemma $2$]{Perko1994} and \cite[Section 4.5]{GuckHolmes}, the displacement map \eqref{4} is constructed for loops, i.e. for polycycles $\Gamma^1$. However, it is clear from the proofs of such results that the hypothesis of having a loop is not necessary. Actually, the existence of the polycycle itself is not necessary. Only the existence of a connection between two saddles, which may be the same.
\end{remark} 

\section{The displacement map between non-subsequent saddles}\label{Sec3}

Under the context of Section~\ref{Sec2.2}, let $\sigma_0\in\{-1,1\}$ be given by $\sigma_0=-1$ (resp. $\sigma_0=1)$ if the first return map associated to $\Gamma^n$ is defined in the inner (resp. outer) region defined by $\Gamma^n$. Suppose $n\geqslant2$. Observe that if $\sigma_0d_n(\mu)>0$, then the intersection 
	\[L_n^u(\mu)\cap l_{n-1}=\left\{x_{n-1}^{(1)}(\mu)\right\},\]
is well defined, see Figure~\ref{Fig2}. Similarly, if $\sigma_0d_{n-1}(\mu)<0$, then 
	\[L_{n-1}^s(\mu)\cap l_n=\left\{x_{n}^{(1)}(\mu)\right\}\]
is also well defined, see Figure~\ref{Fig2x}.
\begin{figure}[ht]
	\begin{center}
		\begin{minipage}{8.5cm}
			\begin{center} 
				\begin{overpic}[width=8cm]{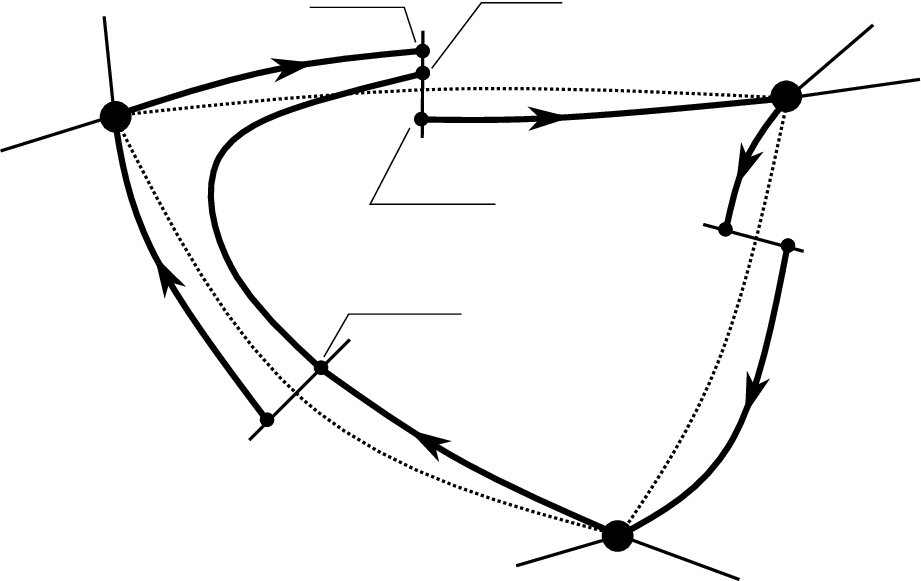} 
					\put(55,40){$x_2^s(\mu)$}
					\put(62,62){$x_2^{(1)}(\mu)$}
					\put(26,12){$x_3^s(\mu)$}
					\put(21,61){$x_2^u(\mu)$}
					\put(51,27.5){$x_3^u(\mu)$}
					
					\put(84,57){$p_2$}
					\put(6,53){$p_3$}
					\put(65,0){$p_1$}				
				\end{overpic}
				
				$\sigma_0=-1$ and $n=3$.
			\end{center}
		\end{minipage}
		\begin{minipage}{8.5cm}
			\begin{center} 
				\begin{overpic}[width=8cm]{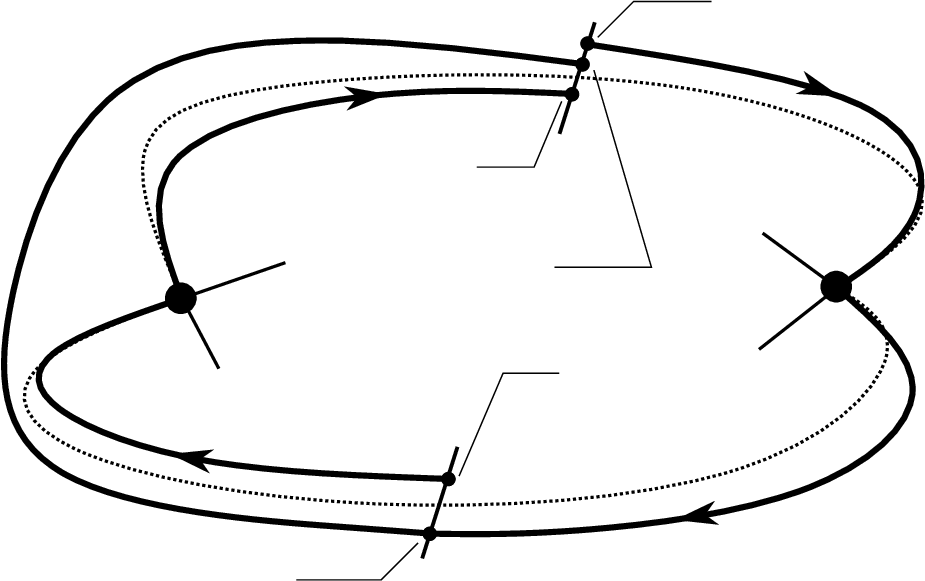} 
					\put(78,62){$x_1^s(\mu)$}
					\put(39,43.5){$x_1^u(\mu)$}
					\put(45,32.5){$x_1^{(1)}(\mu)$}
					\put(62,22){$x_2^s(\mu)$}
					\put(19,-1){$x_2^u(\mu)$}
					
					\put(95,31){$p_1$}
					\put(13,31.5){$p_2$}
				\end{overpic}
				
				$\sigma_0=1$ and $n=2$.
			\end{center}
		\end{minipage}
	\end{center}
	\caption{An illustration of $x_{n-1}^{(1)}$.}\label{Fig2}
	$\;$
	\begin{center}
		\begin{minipage}{8.5cm}
			\begin{center} 
				\begin{overpic}[width=8cm]{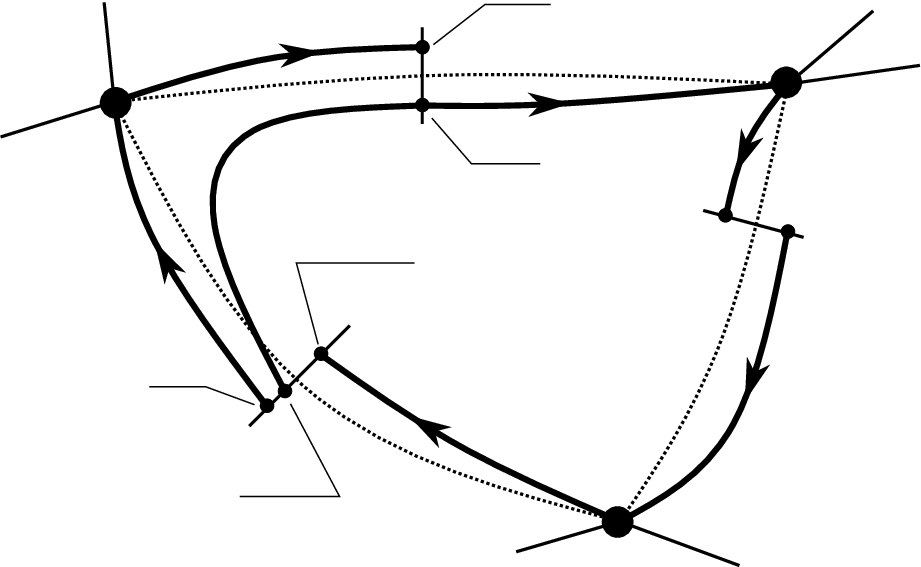} 
					\put(46,32){$x_3^u(\mu)$}
					\put(4,18){$x_3^s(\mu)$}
					\put(11,6){$x_3^{(1)}(\mu)$}
					
					\put(61,60){$x_2^u(\mu)$}
					\put(60,43){$x_2^s(\mu)$}
					
					\put(84,57){$p_2$}
					\put(6,53){$p_3$}
					\put(65,0){$p_1$}				
				\end{overpic}
				
				$\sigma_0=-1$ and $n=3$.
			\end{center}
		\end{minipage}
		\begin{minipage}{8.5cm}
			\begin{center} 
				\begin{overpic}[width=8cm]{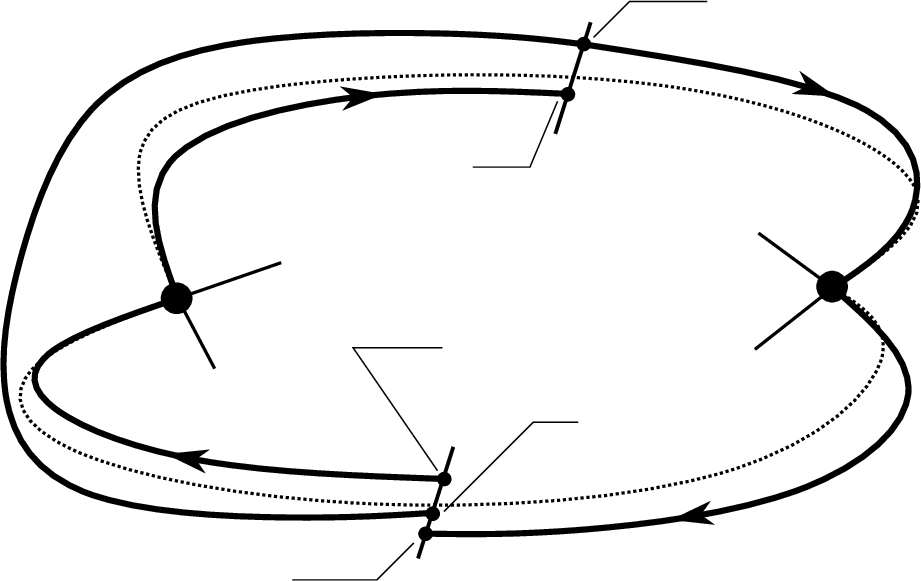} 
					\put(78,62){$x_1^s(\mu)$}
					\put(39,43.5){$x_1^u(\mu)$}
					\put(64,16){$x_2^{(1)}(\mu)$}
					\put(49,24.5){$x_2^s(\mu)$}
					\put(19,-1){$x_2^u(\mu)$}
					
					\put(95,31){$p_1$}
					\put(13,31.5){$p_2$}
				\end{overpic}
				
				$\sigma_0=1$ and $n=2$.
			\end{center}
		\end{minipage}
	\end{center}
	\caption{An illustration of $x_n^{(1)}$.}\label{Fig2x}
\end{figure} 
Similarly to \eqref{17}, let (when well defined) $b_{n-1}^{(1)}(\mu)$, $b_n^{(1)}(\mu)\in\mathbb{R}$ be such that,
\begin{equation}\label{12}
	x_{n-1}^{(1)}(\mu)=x_{n-1}+b_{n-1}^{(1)}(\mu)v_{n-1}, \quad x_n^{(1)}(\mu)=x_n+b_n^{(1)}(\mu)v_n.
\end{equation}
Suppose $r_n>1$. In this case, we define the displacement map $d_{n-1}^{(1)}\colon\Lambda\to\mathbb{R}$ by
\begin{equation}\label{34}
	d_{n-1}^{(1)}(\mu)=\left\{\begin{array}{ll} 
		b_{n-1}^u(\mu)-b_{n-1}^s(\mu), & \text{if } \sigma_0d_n(\mu)\leqslant0, \vspace{0.2cm} \\
		b_{n-1}^{(1)}(\mu)-b_{n-1}^s(\mu), & \text{if } \sigma_0d_n(\mu)>0.
	\end{array}\right.
\end{equation}

\begin{proposition}\label{P2}
	If $r_n>1$, then the displacement map \eqref{34} is a well defined function of class $C^1$. 
\end{proposition}

\begin{proof} For simplicity, assume first that $\sigma_0=-1$, i.e. the displacement map is in the inner region of $\Gamma^n$. In this case, it follows from \eqref{34} that,
\begin{equation}\label{58}
	d_{n-1}^{(1)}(\mu)=\left\{\begin{array}{ll} 
		b_{n-1}^u(\mu)-b_{n-1}^s(\mu), & \text{if } d_n(\mu)\geqslant0, \vspace{0.2cm} \\
		b_{n-1}^{(1)}(\mu)-b_{n-1}^s(\mu), & \text{if } d_n(\mu)<0.
	\end{array}\right.
\end{equation}	
Let $\mu\in\Lambda$ be such that $d_n(\mu)<0$. It follows from \eqref{4} and \eqref{58} that,
\begin{equation}\label{13}
	d_{n-1}^{(1)}(\mu)=d_{n-1}(\mu)+\bigl(b_{n-1}^{(1)}(\mu)-b_{n-1}^u(\mu)\bigr).
\end{equation}  
For every $i\in\{1,\dots,n\}$, let $u_i=-v_i$ (recall from \eqref{55} that $v_i$ points outwards $\Gamma^n$ and thus $u_i$ points inwards). Given $q\in l_n$, observe that there exists a unique $\xi\geqslant0$ such that,
\begin{equation}\label{14}
	q=x_n^s(\mu)+\xi u_n.
\end{equation}
In particular, $\xi=0$ if and only if, $q=x_n^s(\mu)$. If $\xi\geqslant0$ is small enough, then the orbit of $X_{\mu}$ through $q=q(\xi)$ will intersect $l_{n-1}$ in the point $p=p(\xi)$ given 
\begin{equation}\label{18}
	p=x_{n-1}^u(\mu)+D_n(\xi,\mu)u_{n-1},
\end{equation}
where $D_n\colon l_n\times\Lambda\to l_{n-1}$ is the Dulac map associated to $p_n(\mu)$, see Figure~\ref{Fig2}. From \eqref{17}, 
\begin{equation}\label{56}
	x_n^u(\mu)=\bigl(x_n+b_n^s(\mu)v_n\bigr)+\bigl(b_n^u(\mu)-b_n^s(\mu)\bigr)v_n=x_n^s(\mu)+d_n(\mu)v_n,
\end{equation}
and from \eqref{12} we have,
\begin{equation}\label{57}
	x_{n-1}^{(1)}(\mu)=\bigl(x_{n-1}+b_{n-1}^u(\mu)v_{n-1}\bigr)+\bigl(b_{n-1}^{(1)}(\mu)-b_{n-1}^u(\mu)\bigr)v_{n-1}=x_{n-1}^u(\mu)+\bigl(b_{n-1}^{(1)}(\mu)-b_{n-1}^u(\mu)\bigr)v_{n-1}.
\end{equation}
Hence, if we let $q=x_n^u(\mu)$, then it follows from \eqref{14} and \eqref{56} that $\xi=-d_n(\mu)$ (recall that $d_n(\mu)<0$). Since $x_{n-1}^{(1)}(\mu)$ is the intersection of the positive orbit through $x_n^u(\mu)$ with $l_{n-1}$, it follows from \eqref{18} and \eqref{57} that,
	\[b_{n-1}^{(1)}(\mu)-b_{n-1}^u(\mu)=-D_n(-d_n(\mu),\mu).\]
Therefore, it follows from \eqref{13} that if $d_n(\mu)<0$, then
\begin{equation}\label{15}
	d_{n-1}^{(1)}(\mu)=d_{n-1}(\mu)-D_n(-d_n(\mu),\mu).
\end{equation}
Hence, as a consequence of \eqref{58} and \eqref{15} we arrive to
	\[d_{n-1}^{(1)}(\mu)=d_{n-1}(\mu)+R(\mu),\]
where $R\colon\Lambda\to\mathbb{R}$ is given by
	\[R(\mu)=\left\{\begin{array}{ll} 
		0, & \text{if } d_n(\mu)\geqslant0, \vspace{0.2cm} \\
		-D_n(-d_n(\mu),\mu), & \text{if } d_n(\mu)<0.
	\end{array}\right.\]
Since $r_n>1$, it follows from Proposition~\ref{MV1} and Remark~\ref{Remark1} that if $d_n(\mu)<0$, then
\begin{equation}\label{60}
	\frac{\partial R}{\partial\mu_j}(\mu)=\frac{\partial D_n}{\partial s}(-d_n(\mu),\mu)\frac{\partial d_n}{\partial\mu_j}(\mu)-\frac{\partial D_n}{\partial\mu_j}(-d_n(\mu),\mu).
\end{equation}
Hence, if we take $\mu_0\in d_n^{-1}(\{0\})$, then from \eqref{60} and Proposition~\ref{MV1} we know that,
	\[\lim\limits_{\mu\to\mu_0}\frac{\partial R}{\partial\mu_j}(\mu)=0, \quad \lim\limits_{\mu\to\mu_0}R(\mu)=0.\]
Therefore, $R$ is of class $C^1$ and thus $d_{n-1}^{(1)}$ is also $C^1$. The proof for the case $\sigma_0=1$ follows similarly. We only observe that in general we have $u_i=\sigma_0v_i$ and
	\[d_{n-1}^{(1)}(\mu)=d_{n-1}(\mu)+R_{\sigma_0}(\mu),\]
where,
	\[R_{\sigma_0}(\mu)=\left\{\begin{array}{ll} 
		0, & \text{if } \sigma_0d_n(\mu)\leqslant0, \vspace{0.2cm} \\
		\sigma_0D_n(\sigma_0d_n(\mu),\mu), & \text{if } \sigma_0d_n(\mu)>0.
	\end{array}\right.\]
This finishes the proof. \end{proof}

\begin{remark}\label{Remark3}
	If $r_n>1$, then it follows from \eqref{34} that if $\sigma_0d_n(\mu)>0$ and $d_{n-1}^{(1)}(\mu)=0$, then we have a heteroclinic (or homoclinic, if $n=2$) connection from $p_1$ to $p_{n-1}$, bypassing $p_n$. In particular, if there is $\mu_0\in\Lambda$ such that $\sigma_0d_n(\mu_0)>0$, then from the continuity of $d_n$ we know that there is a neighborhood $\Lambda_1\subset\Lambda$ of $\mu_0$ such that $\sigma_0d_n(\mu)>0$ for every $\mu\in\Lambda_1$. Hence, if $\mu\in\Lambda_1$ is such that $d_{n-1}^{(1)}(\mu)=0$, then there is a connection from $p_1$ to $p_{n-1}$, bypassing $p_n$.
\end{remark}

Suppose now that $r_n<1$. In this case, we define the displacement map $d_{n-1}^{(1)}\colon\Lambda\to\mathbb{R}$ by,
\begin{equation}\label{59}
	d_{n-1}^{(1)}(\mu)=\left\{\begin{array}{ll} 
		b_{n}^u(\mu)-b_{n}^s(\mu), & \text{if } \sigma_0d_{n-1}(\mu)\geqslant0, \vspace{0.2cm} \\
		b_{n}^u(\mu)-b_{n}^{(1)}(\mu), & \text{if } \sigma_0d_{n-1}(\mu)<0.
	\end{array}\right.
\end{equation}

\begin{proposition}\label{P2x}
	If $r_n<1$, then the displacement map \eqref{59} is a well defined function of class $C^1$. 
\end{proposition}

\begin{proof} The proof follows similarly to the proof of Proposition~\ref{P2}. We only observe that in this case we have,
	\[d_{n-1}^{(1)}(\mu)=d_n(\mu)+R_{\sigma_0}^*(\mu),\]
with,
	\[R_{\sigma_0}^*(\mu)=\left\{\begin{array}{ll} 
		0, & \text{if } \sigma_0d_{n-1}(\mu)\geqslant0, \vspace{0.2cm} \\
		\sigma_0D_n^{-1}(\sigma_0d_{n-1}(\mu),\mu), & \text{if } \sigma_0d_{n-1}(\mu)<0,
	\end{array}\right.\]
where $D_n$ is the displacement map associated to $p_n$. \end{proof}

In case $r_n<1$ a remark similar to Remark~\ref{Remark3} could be done.

\begin{corollary}\label{Coro1}
	If $r_n\neq1$, then
		\[\frac{\partial d_{n-1}^{(1)}}{\partial\mu_j}(0)=\left\{\begin{array}{ll} 
			\displaystyle\frac{\partial d_{n-1}}{\partial\mu_j}(0), & \text{if } r_n>1, \vspace{0.2cm} \\
			\displaystyle\frac{\partial d_{n}}{\partial\mu_j}(0), & \text{if } r_n<1,
		\end{array}\right.\]
	for every $j\in\{1,\dots,N\}$.
\end{corollary}

\begin{remark}\label{Remark5} As a consequence of Propositions~\ref{P2} and \ref{P2x},  for $d_{n-1}^{(1)},$ defined as in  
\eqref{34} or \eqref{59}, respectively, to be of class $C^1$  it is sufficient to have $r_n\neq1$. If $r_n=1$, then it follows from \cite{MarVil2021} that we can write
		\[D(s,\mu)=\Delta_{00}(\mu)s^{r(\mu)}+\mathscr{R}(s,\mu),\]
	with $\mathscr{R}\colon[0,\varepsilon)\times\Lambda\to\mathbb{R}$ continuous, $C^1$ in $(0,\varepsilon)\times\Lambda$ and such that
		\[\mathscr{R}(0,\mu)=0, \quad \lim\limits_{s\to0^+}\frac{\partial\mathscr{R}}{\partial\mu_j}(0,\mu)=0,\]
	for every $\mu\in\Lambda$ and $j\in\{1,\dots,N\}$. In particular, $D$ can be continuously extended to $s=0$ and, in relation to the parameter $\mu$, can also be $C^1$-extended. However, such $C^1$-extension does not necessarily hold in relation to $s$ and thus we cannot in general apply the limit in \eqref{60}. Hence, if $r_n=1$, then \eqref{34} and \eqref{59} are both well defined continuous functions, but not necessarily of class $C^1$.
\end{remark}

\begin{remark}\label{Remark7}
	Our construction of the map $d_{n-1}^{(1)}$ given in this section is inspired on the construction of the map $d_{n-1}^*$ of \cite[Lemma $2.2$]{HanWuBi2004}. However, in the construction given there  the authors do not give the proof that the Dulac map in relation to the perturbative parameter $\mu$ is continuously differentiate. We prove this regularity by using the recent works~\cites{MarVil2020,MarVil2021,MarVil2024} that give properties of the Dulac map not know at that time. Moreover, when defining $d_{n-1}^*$ the authors in \cite{HanWuBi2004} seem not to be aware that the points $x_{n-1}^{(1)}(\mu)$ and $x_{n}^{(1)}(\mu)$ are not well defined for every $\mu\in\Lambda$. Hence, in our proof we need to define $d_{n-1}^{(1)}$ in a multiple-folded way, depending on the sign of $\sigma_0$, $d_n(\mu)$ and $d_{n-1}(\mu)$. 
\end{remark}

\section{Some technical results}\label{Sec4}

\subsection{Polynomial approximation of a bump function}\label{Sec4.1}

Let $F\colon [0,1]^2\to\mathbb{R}$ be a map of class $C^r$, $r\geqslant0$. The \emph{Bernstein polynomial} associated to $F$ is given by
	\[B_{m,n}^F(x_1,x_2)=\sum_{r=0}^{m}\sum_{s=0}^{n}F\left(\frac{r}{m},\frac{s}{n}\right)\binom{m}{r}\binom{n}{s}x_1^rx_2^s(1-x_1)^{m-r}(1-x_2)^{n-s},\]
where $\binom{n}{k}=\frac{n!}{k!(n-k)!}$. An important property of the Bernstein polynomials is that $B_{n,m}\rightrightarrows F$ uniformly in the $C^r$-topology. More precisely, we have the following theorem (see  Kingsley~\cite{King1949}).

\begin{proposition}\label{King}
	If $F\colon [0,1]^2\to\mathbb{R}$ is of class $C^r$, $r\geqslant0$ finite, then
		\[\lim\limits_{(n,m)\to\infty}\frac{\partial^{|k|}B_{m,n}^F}{\partial x_1^{k_1}\partial x_2^{k_2}}(x_1,x_2)=\frac{\partial^{|k|}F}{\partial x_1^{k_1}\partial x_2^{k_2}}(x_1,x_2),\]
	uniformly in $(x_1,x_2)\in[0,1]^2$, where $k=(k_1,k_2)\in\mathbb{Z}_{\geqslant0}^2$, $|k|=k_1+k_2$ and $|k|\leqslant r$. 
\end{proposition}

In particular, we can use Proposition~\ref{King} to construct suitable polynomial approximations of a given bump function. More precisely, given $\delta_2>\delta_1>0$ and $c\in\mathbb{R}^2$, we say that a $C^\infty$-function $\varphi\colon\mathbb{R}^2\to[0,1]$ is a $(\delta_1,\delta_2,c)$-bump function if
	\[\varphi(x)=\left\{\begin{array}{ll} 1, & \text{if } ||x-c||\leqslant \delta_1, \vspace{0.2cm} \\ 0, & \text{if } ||x-c||\geqslant\delta_2. \end{array}\right.\]
	
\begin{proposition}\label{P1}
	Set $r\geqslant0$ finite and let $\varphi\colon\mathbb{R}^2\to[0,1]$ be a $(\delta_1,\delta_2,c)$-bump function. Then for every compact $B\subset\mathbb{R}^2$ and $\varepsilon>0$, there is a polynomial $q\colon\mathbb{R}^2\to\mathbb{R}$ such that 
	\begin{equation}\label{19}
		\max_{\substack{x\in B \\ |k|\leqslant r}}\left|\frac{\partial^{|k|}\varphi}{\partial x_1^{k_1}\partial x_2^{k_2}}(x)-\frac{\partial^{|k|}q}{\partial x_1^{k_1}\partial x_2^{k_2}}(x)\right|<\varepsilon,
	\end{equation}
	where $k=(k_1,k_2)\in\mathbb{Z}^2_{\geqslant0}$ and $|k|=k_1+k_2$. Moreover, $q$ can be chosen such that 
	\begin{equation}\label{20}
		\varphi(x)+\frac{1}{4}\varepsilon<q(x)<\varphi(x)+\frac{3}{4}\varepsilon,
	\end{equation}
	for every $x\in B$. In particular, $q(x)>0$ for every $x\in B$.
\end{proposition}

\begin{proof} Except by a translation and a linear change of coordinates, we can suppose $B\subset[0,1]^2$. It follows from Proposition~\ref{King} that there is a polynomial $\overline{q}\colon\mathbb{R}^2\to\mathbb{R}$ such that,
\begin{equation}\label{21}
	\max_{\substack{x\in[0,1]^2 \\ |k|\leqslant r}}\left|\frac{\partial^{|k|}\varphi}{\partial x_1^{k_1}\partial x_2^{k_2}}(x)-\frac{\partial^{|k|}\overline{q}}{\partial x_1^{k_1}\partial x_2^{k_2}}(x)\right|<\frac{1}{4}\varepsilon.
\end{equation}
 Consider now the polynomial $q\colon\mathbb{R}^2\to\mathbb{R}$ given by $q(x)=\overline{q}(x)+\frac{1}{2}\varepsilon$. We claim that it satisfies \eqref{19} and \eqref{20}. Indeed, since $q$ is a translation of $\overline{q}$ it follows that its partial derivatives of order $|k|$, $|k|\geqslant 1$, are equal. Therefore \eqref{19} follows directly from \eqref{21}, when $|k|\geqslant 1$. We now look the case $|k|=0$, i.e. the function $q$ itself. It follows from \eqref{21}, with $|k|=0$, that
\begin{equation}\label{22}
	\varphi(x)-\frac{1}{4}\varepsilon<\overline{q}(x)<\varphi(x)+\frac{1}{4}\varepsilon.
\end{equation}
Adding $\frac{1}{2}\varepsilon$ on all sides of \eqref{22} we obtain
\begin{equation}\label{24}
	\varphi(x)+\frac{1}{4}\varepsilon<q(x)<\varphi(x)+\frac{3}{4}\varepsilon.
\end{equation}
which is precisely \eqref{20}. Moreover, it also follows from \eqref{24} that
	\[\varphi(x)-\varepsilon<q(x)<\varphi(x)+\varepsilon,\]
and thus we have \eqref{19} with $|k|=0$, completing the proof. \end{proof}

\begin{remark}\label{Remark4}
	Let $\varphi\colon\mathbb{R}^2\to[0,1]$ be a $(\delta_1,\delta_2,c)$-bump, function and set $\overline{\varepsilon}>0$ and $\varepsilon>0$ such that $\varepsilon\leqslant\frac{1}{3}\overline{\varepsilon}$. Let $q_{\overline{\varepsilon}}$ and $q_\varepsilon$ be the respective polynomials given by Proposition~\ref{P1}. It follows from \eqref{24} that
		\[\varphi(x)<q_\varepsilon(x)<\varphi(x)+\frac{3}{4}\varepsilon\leqslant\varphi(x)+\frac{1}{4}\overline{\varepsilon}<q_{\overline{\varepsilon}}(x).\]
	In particular, $\varphi(x)<q_{\overline{\varepsilon}}(x)$ and $q_\varepsilon(x)<q_{\overline{\varepsilon}}(x)$ for all $x\in B$.
\end{remark}

\subsection{Positive or negative invariant regions associated to a simple polycycle}\label{Sec4.2}

Let $X$ be a planar smooth vector field and $\Omega\subset\mathbb{R}^2$ an open set. We say that $\Omega$ is \emph{positive-invariant} (resp. \emph{negative-invariant}) by $X$ if for every $x\in\Omega$ we have $\gamma(t)\in\Omega$ for all $t\geqslant0$ (resp. $t\leqslant0$), where $\gamma(t)$ is the orbit of $X$ with initial condition $\gamma(0)=x$.

Let $S\subset\mathbb{R}^2$ be a continuous simple closed curve. We say that $S$ is \emph{piecewise smooth} if it is of class $C^\infty$ except, perhaps, in at most a finite number of points. We will say that a piecewise smooth closed curve is without contact with a smooth vector field if on each of the closed $C^\infty$ sides of $S,$ the scalar product $\left<X, \nabla S\right>$   keeps sign on all the regular points of $S,$ and on $S$ either $X$  points always towards the interior of the region delimited by $S$ or  $X$  points always towards the exterior of this region. 

The proof of next result follows mutatis mutandis the proof of a similar result, but with an isolated limit cycle instead of a polycycle, see \cite[Proposition $1$]{Santana}. We omit the details. For an illustration of the situation  see Figure~\ref{Fig7}.  As usual, given a compact set $B\subset\mathbb{R}^2$, let $\textnormal{Int}(B)$ be its topological interior.

\begin{proposition}\label{P3}
	Let $\mathcal{X}$ be one of the topological spaces $\mathfrak{X}^\infty$ or $\mathcal{P}^r$, for some $r\geqslant1$. Let $X\in\mathcal{X}$ having a simple polycycle $\Gamma^n$ composed by $n\geqslant1$ hyperbolic saddles and let $B\subset\mathbb{R}^2$ be a compact set such that $\Gamma^n\subset\operatorname{Int}(B).$ Then there is a continuous and piecewise smooth simple closed curve $S\subset\operatorname{Int}(B)$, on the same connected component of $B\backslash\Gamma^n$ as the first-return map of $\Gamma^n$, such that if $\Omega\subset\operatorname{Int}(B)$ is the open region bounded by $S$ and $\Gamma^n$, then following statements hold.
	\begin{enumerate}[label=(\alph*)]
		\item There is no singularity of $X$ in $\Omega$.
		\item There is no periodic orbit of $X$ in $\Omega$.
		\item $X$ is without contact with $S$.
		\item If $r(\Gamma^n)>1$, then $\Omega$ is positive invariant by $X$.
		\item If $r(\Gamma^n)<1$, then $\Omega$ is negative invariant by $X$.
	\end{enumerate}
\end{proposition}

\begin{remark}
	Under the statement of Proposition~\ref{P3}, it follows from the compactness of $S$ and the continuity of the inner product $\left<\cdot,\cdot\right>$ that there is a neighborhood $N\subset\mathcal{X}$ of $X$ such that $\left<X(s),Y(s)\right>>0$ for every $Y\in N$ and $s\in S$. In particular, $Y$ is also without contact with $S$ and points in the same direction as $X$.
\end{remark}

\begin{figure}[h]
	\begin{center}
		\begin{minipage}{7cm}
			\begin{center} 
				\begin{overpic}[height=4cm]{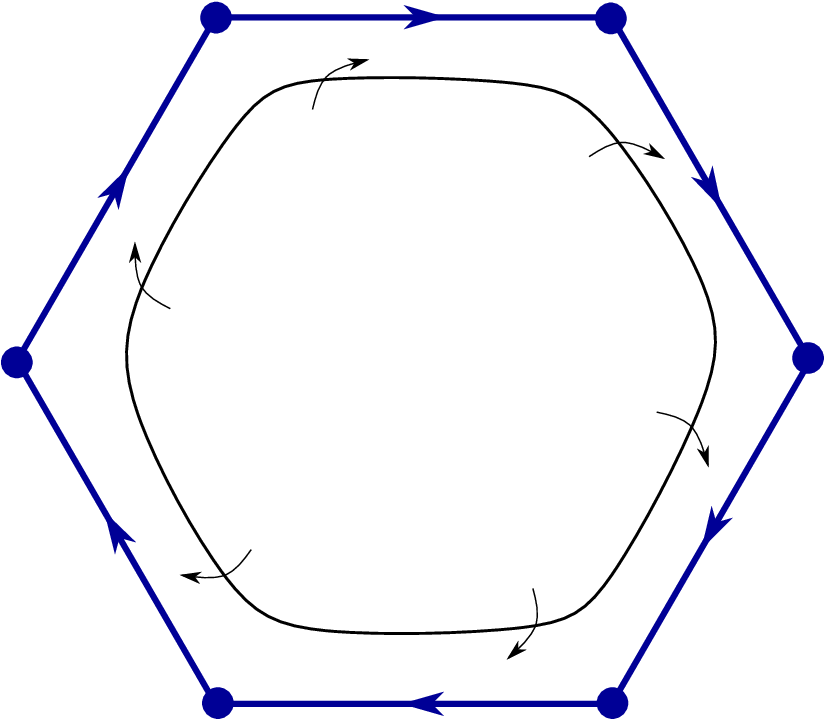} 
					\put(22,30){$S$}	
					\put(85,75){$\Gamma^n$}					
				\end{overpic}
				
				$S$ in the bounded region of $\Gamma^n$.
			\end{center}
		\end{minipage}
		\begin{minipage}{7cm}
			\begin{center} 
				\begin{overpic}[height=4cm]{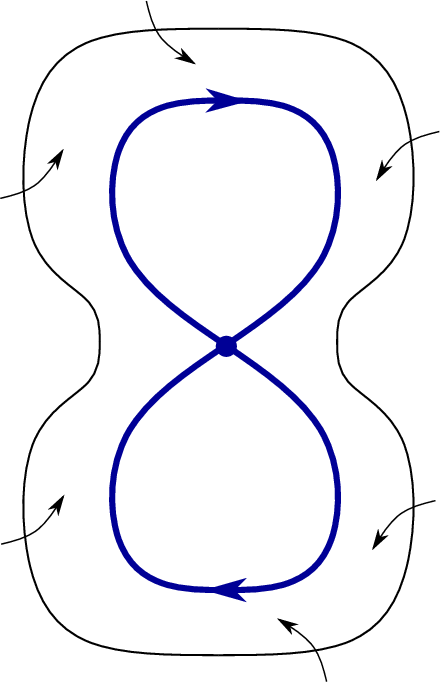}
					\put(59,90){$S$}
					\put(20,25){$\Gamma^n$}
				\end{overpic}
				
				$S$ in the unbounded region of $\Gamma^n$.
			\end{center}
		\end{minipage}
	\end{center}
	\caption{Illustration of of the curve $S$ and the flow of $X$ on it, for the case $r(\Gamma^n)>1$.}\label{Fig7}
\end{figure} 

\subsection{Periodic orbits of smooth vector fields}\label{Sec4.3}

Let $X$ be a planar smooth vector field with a periodic orbit $\gamma(t)$ (not necessarily isolated), with period $T>0$. It follows from Andronov et al \cite[Lemma $1$, p. $124$]{And1971} that there is a neighborhood $A\subset\mathbb{R}^2$ of $\gamma$ and a smooth function $\Phi\colon A\to\mathbb{R}$ such that
\begin{equation}\label{42}
	\Phi(\gamma(t))=0, \quad \left(\frac{\partial\Phi}{\partial x}(\gamma(t))\right)^2+\left(\frac{\partial\Phi}{\partial y}(\gamma(t))\right)^2>0,
\end{equation}
for every $t\in[0,T]$. In particular, by means of \emph{bump-functions} we can suppose that $\Phi$ is defined on the entire plane and has compact support. The authors in \cite[Theorem $19$]{And1971} use $\Phi$ to perturb the stability of non-hyperbolic limit cycles, bifurcating new ones in the process. In the next result we enunciate and proof a simple version of their results, sufficient for our objectives in this paper.

\begin{proposition}\label{P4}
	Let $X=(P,Q)$, $\gamma$ and $\Phi$ be as above and consider the one-parameter family of planar smooth vector fields $X_\lambda=(P_\lambda,Q_\lambda)$ given by,
	\begin{equation}\label{44}
		P_\lambda(x,y)=P(x,y)+\lambda\Phi(x,y)\frac{\partial\Phi}{\partial x}(x,y), \quad Q_\lambda(x,y)=Q(x,y)+\lambda\Phi(x,y)\frac{\partial\Phi}{\partial y}(x,y),
	\end{equation}
	with $\lambda\in\mathbb{R}$. Then if $\gamma$ is not a hyperbolic limit cycle for $X$ then it is a hyperbolic limit cycle for every $\lambda\neq0$ and its stability depends on the sign of $\lambda$. Otherwise, it is hyperbolic for $|\lambda|$ small enough.
\end{proposition}

\begin{proof} It follows from \eqref{42} that $X_\lambda(\gamma(t))=X(\gamma(t))$ for every $t\in[0,T]$ and $\lambda\in\mathbb{R}$. Hence, $\gamma(t)$ is also a periodic orbit of $X_\lambda$. The first derivative of the Poincar\'e first return map of $X_\lambda$ at $\gamma$ is given by
	\[r(\lambda):=\int_{0}^{T}\left(\frac{\partial P_\lambda}{\partial x}+\frac{\partial Q_\lambda}{\partial y}\right)(\gamma(t))\;dt,\]
see for example \cite[Theorem $1.23$]{DumLliArt2006}. It follows from the expression \eqref{44} of $X_\lambda$ that
	\[r(\lambda)=\int_{0}^{T}\frac{\partial P}{\partial x}+\frac{\partial Q}{\partial y}\;dt+\lambda\int_{0}^{T}\Phi\frac{\partial^2\Phi}{\partial x^2}+\Phi\frac{\partial^2\Phi}{\partial y^2}\;dt+\lambda\int_{0}^{T}\left(\frac{\partial\Phi}{\partial x}\right)^2+\left(\frac{\partial\Phi}{\partial y}\right)^2\;dt,\]
where term $\gamma(t)$ was omitted by simplicity. Observe that the middle integral of the right-hand side is equal to zero because $\Phi(\gamma)\equiv0.$  Thus
	\[r(\lambda)=r(0)+\lambda\int_{0}^{T}\left(\frac{\partial\Phi}{\partial x}(\gamma(t))\right)^2+\left(\frac{\partial\Phi}{\partial y}(\gamma(t))\right)^2\;dt.\]
From \eqref{42} we have that the above integral is positive and thus if $r(0)=0$ (i.e. $\gamma$ is not hyperbolic for $X$) then $r(\lambda)\neq0$ for every $\lambda\neq0$. In particular, $\text{sign}(r(\lambda))=\text{sign}(\lambda)$ and thus we can choose the stability of $\gamma$.  If $\gamma$ is a hyperbolic limit cycle for $X$ then $r(0)\neq0$ and thus $\gamma$ remains a hyperbolic limit cycle of same stability for $|\lambda|$ small enough. \end{proof} 

\section{Proof of Theorem \ref{Main1}}\label{Sec5}

For simplicity, we assume for now that $\Gamma^n$ is endowed with the trivial permutation $\tau$. We recall that $R_i=\prod_{j=1}^{i}r_j$, where $r_i$ are the hyperbolicity ratios \eqref{3} of the hyperbolic saddles of the polycycle $\Gamma^n$. Let $L_i$ and $x_i$ be as in Section~\ref{Sec2.2}. For each $i\in\{1,\dots,n\}$, let $\gamma_i(t)$ be the parametrization of $L_i$ given by the solution of $X$ and with the initial condition $\gamma_i(0)=x_i$. Let also $L_i^+=\{\gamma_i(t)\colon t>0\}$. Let $B\subset\mathbb{R}^2$ be a closed ball such that $\Gamma^n\subset\operatorname{Int}(B)$. For each $i\in\{1,\dots,n\}$, let $c_i\in L_i^+$ and let $\delta_{i,2}>\delta_{i,1}>0$ be small enough such that the compact sets
	\[G_{i,j}=\{(x_1,x_2)\in\mathbb{R}^2\colon ||x-c_i||\leqslant\delta_{i,j}\},\]
satisfies the following statements.	
\begin{enumerate}[label=(\alph*)]
	\item $\Gamma^n\cap G_{i,j}=L_i^+\cap G_{i,j}\neq\emptyset$, $j\in\{1,2\}$;
	\item If $i\neq k$, then $G_{i,2}\cap G_{k,2}=\emptyset$;
	\item $G_{i,j}\subset\operatorname{Int}(B)$.
\end{enumerate}
See Figure~\ref{Fig3}.
\begin{figure}[ht]
	\begin{center}
		\begin{minipage}{8.5cm}
			\begin{center} 
				\begin{overpic}[height=6cm]{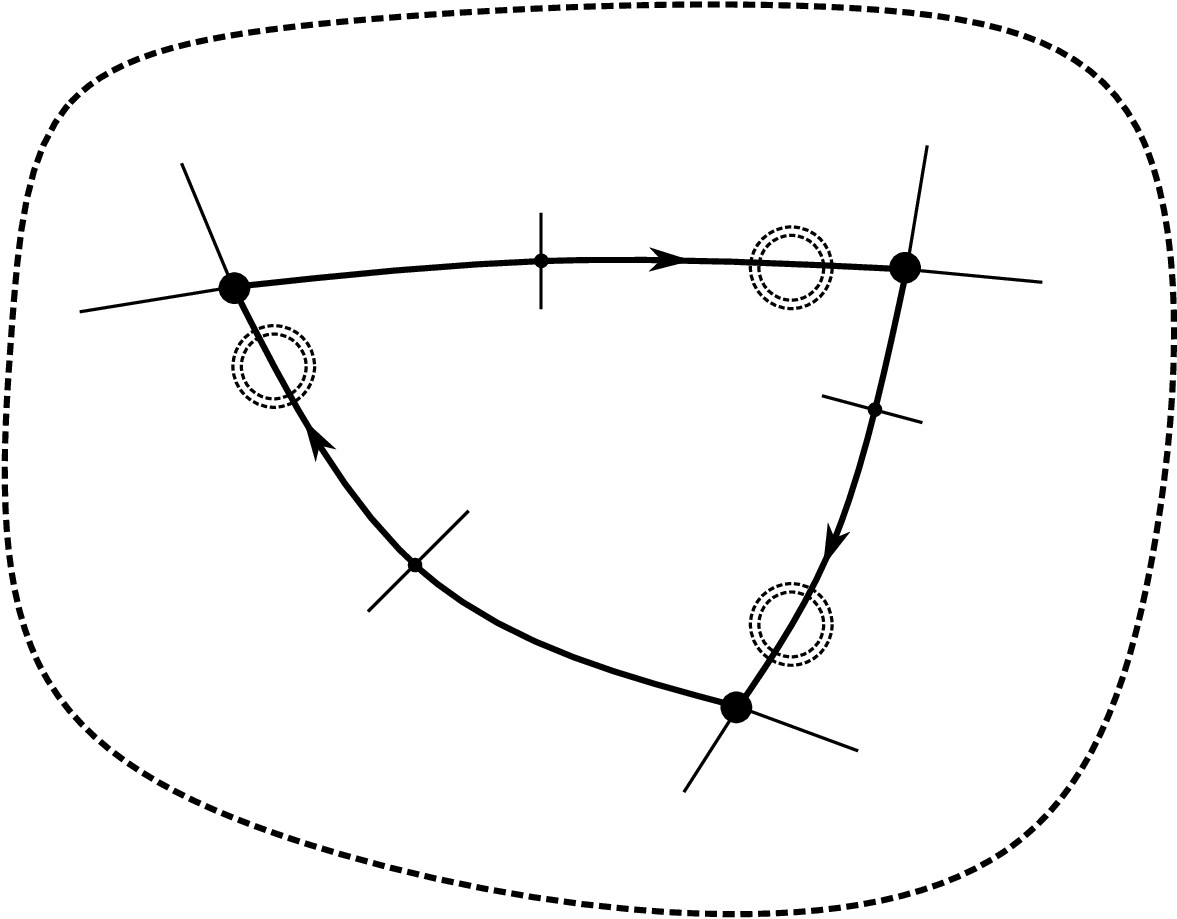} 
					\put(79,57){$p_1$}
					\put(12,55){$p_2$}
					\put(62,13){$p_3$}	
					\put(95,72){$B$}
					\put(60,48){$G_{1,j}$}
					\put(28,45){$G_{2,j}$}
					\put(72,22){$G_{3,j}$}
					\put(38,57){$L_1$}
					\put(38,20){$L_2$}
					\put(74,35){$L_3$}			
				\end{overpic}
			\end{center}
		\end{minipage}
		\begin{minipage}{8.5cm}
			\begin{center} 
				\begin{overpic}[height=6cm]{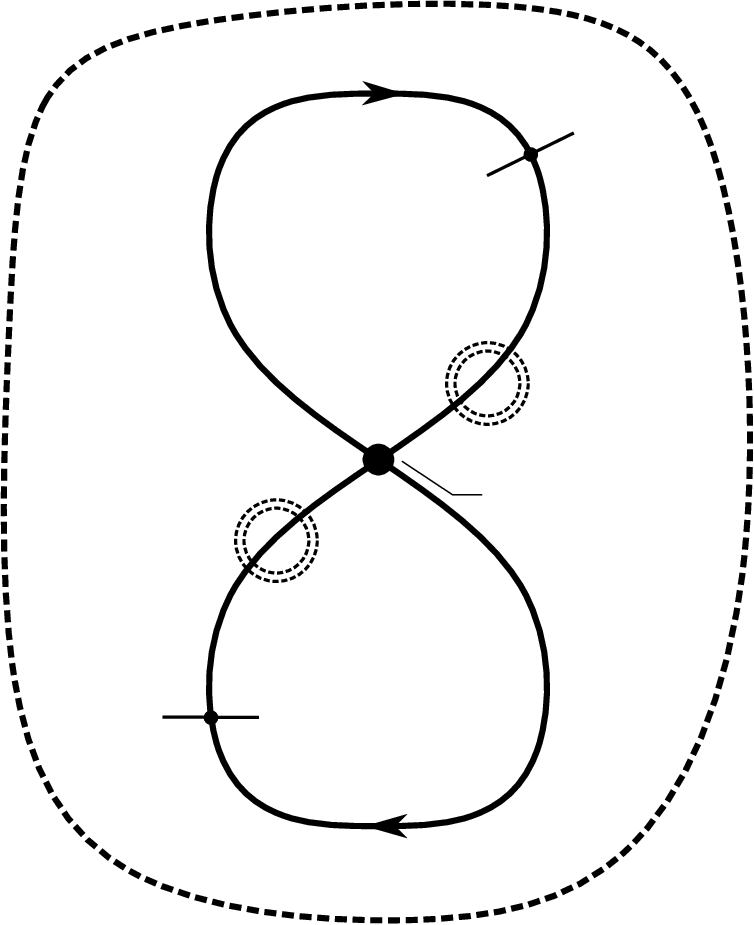} 
					\put(54,45){$p_1=p_2$}
					\put(26,80){$L_1$}
					\put(49,18){$L_2$}
					\put(72,95){$B$}
					\put(59,56){$G_{1,j}$}
					\put(13,40){$G_{2,j}$}	
				\end{overpic}
			\end{center}
		\end{minipage}
	\end{center}
	\caption{An illustration of the sets $G_{i,j}$.}\label{Fig3}
\end{figure} 
Let $\varphi_i\colon\mathbb{R}^2\to\mathbb{R}$ be a $(\delta_{i,1},\delta_{i,2},c_i)$-bump function. Given $\varepsilon>0$, let $q_{i,\varepsilon}\colon\mathbb{R}^2\to\mathbb{R}$ be the polynomial approximation of $\varphi_i$, on $B$, given by Proposition~\ref{P1}. Write $X=(P,Q)$ and let $X^\perp=(-Q,P)$. Let $K\colon\mathbb{R}^2\times\mathbb{R}^n\times(0,+\infty)\to\mathbb{R}^2$ be given by
\begin{equation}\label{9}
	K(x,\mu,\varepsilon)=\left(\sum_{i=1}^{n}\mu_iq_{i,\varepsilon}(x)\right)X^\perp(x),
\end{equation}
and denote
\begin{equation}\label{10}
	X_{\mu,\varepsilon}(x)=X(x)+K(x,\mu,\varepsilon).
\end{equation}
It follows from Proposition~\ref{P1} that $q_{i,\varepsilon}\to\varphi_i$ in the $C^r$-topology (restricted to $B$) as $\varepsilon\to0$. Hence we also let $K\colon\mathbb{R}^2\times\mathbb{R}^n\times\{0\}\to\mathbb{R}^2$ be given by
\begin{equation}\label{2}
	K(x,\mu,0)=\left(\sum_{i=1}^{n}\mu_i\varphi_i(x)\right)X^\perp(x),
\end{equation}
and denote 
\begin{equation}\label{16}
	X_{\mu,0}(x)=X(x)+K(x,\mu,0).
\end{equation}
Observe $X_{0,\varepsilon}=X$ for every $\varepsilon\geqslant0$ and that \emph{for each fixed} $\varepsilon\geqslant0$ the family $X_{\mu,\varepsilon}$ is a well defined family of $C^\infty$-vector fields containing $X$, relative to the parameter $\mu\in\Lambda$. Moreover it is also polynomial if $X$ is polynomial and $\varepsilon>0$. In other words, $X_{\mu,\varepsilon}$ is rather a one-parameter family of one-parameter families of vector fields $(X_\mu)_\varepsilon$, than a two-parameter family. However for simplicity we write $X_{\mu,\varepsilon}$.

Observe that $X_{\mu,0}\to X$ in $\mathfrak{X}^\infty$ as $\mu\to0$ and that if $X$ is polynomial, then given any neighborhood $N\subset\mathcal{P}^r$ of $X$ we can extend the compact $B\subset\mathbb{R}^2$ if necessary such that $X_{\mu,\varepsilon}\in N$ for every $(\mu,\varepsilon)\approx(0,0)$, $\varepsilon>0$. Let $\Lambda\subset\mathbb{R}^n$ be a small enough neighborhood of the origin and let $\overline{\varepsilon}>0$ be small enough. Observe that for each $\varepsilon\in(0,\overline{\varepsilon}]$ we have $X_{\mu,\varepsilon}\to X_{0,\varepsilon}=X$ in $\mathfrak{X}^\infty$ restricted to $B$ (and in particular in a neighborhood of $\Gamma^n$), as $\mu\to0$. Hence it follows that for each $\varepsilon\in(0,\overline{\varepsilon}]$ the displacement maps $d_{i,\varepsilon}\colon\Lambda\to\mathbb{R}$, $i\in\{1,\dots,n\}$, are well defined and of class $C^\infty$. Moreover from \eqref{5} we get that,
\begin{equation}\label{6}
		\frac{\partial d_{i,\varepsilon}}{\partial \mu_j}(0)=\frac{1}{||X(x_i)||}\int_{-\infty}^{+\infty}e^{-\int_{0}^{t}\operatorname{div} X(\gamma_i(s))\;ds}X(\gamma_i(t))\land\frac{\partial K}{\partial\mu_j}(\gamma_i(t),0,\varepsilon)\;dt.
\end{equation}
In particular, the improper integrals in the right hand-side of \eqref{6} are convergent. Similarly, it follows that for $\varepsilon=0$ the displacement maps $d_{i,0}\colon\Lambda\to\mathbb{R}$, $i\in\{1,\dots,n\}$, are also well defined $C^\infty$-maps and their partial derivatives are given by,
\begin{equation}\label{1}
	\frac{\partial d_{i,0}}{\partial \mu_j}(0)=\frac{1}{||X(x_i)||}\int_{-\infty}^{+\infty}e^{-\int_{0}^{t}\operatorname{div} X(\gamma_i(s))\;ds}X(\gamma_i(t))\land\frac{\partial K}{\partial\mu_j}(\gamma_i(t),0,0)\;dt.
\end{equation}
We claim that
\begin{equation}\label{45}
	\lim\limits_{\varepsilon\to0}\frac{\partial d_{i,\varepsilon}}{\partial\mu_j}(0)=\frac{\partial d_{i,0}}{\partial\mu_j}(0),
\end{equation}
for every $i$, $j\in\{1,\dots,n\}$. Indeed, observe that if $\varepsilon>0$, then it follows from \eqref{9} and \eqref{10} that,
\begin{equation}\label{35}
	X(x)\land\frac{\partial K}{\partial\mu_j}(x,0,\varepsilon)=(P,Q)\land(-q_{j,\varepsilon}Q,q_{j,\varepsilon}P)=q_{j,\varepsilon}(P^2+Q^2).
\end{equation}
Similarly, observe that if $\varepsilon=0$, then it follows from \eqref{2} and \eqref{16} that,
\begin{equation}\label{11}
	X(x)\land\frac{\partial K}{\partial\mu_j}(x,0,0)=\varphi_j(P^2+Q^2).
\end{equation}
For each $\varepsilon\in[0,\overline{\varepsilon}]$ and $i$, $j\in\{1,\dots,n\}$, let
	\[\Phi^{i,j}_\varepsilon(t)=e^{-\int_{0}^{t}\operatorname{div} X(\gamma_i(s))\;ds}X(\gamma_i(t))\land\frac{\partial K}{\partial\mu_j}(\gamma_i(t),0,\varepsilon)\]
be the integrand of the right-hand side of \eqref{6} and \eqref{1}. From Proposition~\ref{P1} we know that $q_{j,\varepsilon}>\varphi_j\geqslant0$ and thus from \eqref{35} and \eqref{11} we have that $\Phi^{i,j}_\varepsilon(t)\geqslant0$ for each $t\in\mathbb{R}$, $\varepsilon\in[0,\overline{\varepsilon}]$ and $i$, $j\in\{1,\dots,n\}$.

From Remark~\ref{Remark4} we have that $\Phi^{i,j}_\varepsilon$, with $\varepsilon\in[0,\frac{1}{3}\overline{\varepsilon}]$, is dominated by $\Phi^{i,j}_{\overline{\varepsilon}}$ (i.e $|\Phi^{i,j}_\varepsilon(t)|\leqslant\Phi^{i,j}_{\overline{\varepsilon}}(t)$, for each $t\in\mathbb{R}$). Moreover since \eqref{6} is well defined for $\varepsilon=\overline{\varepsilon}$, it follows that
	\[\int_{-\infty}^{+\infty}\Phi^{i,j}_{\overline{\varepsilon}}(t)\;dt<\infty,\]
for $i$, $j\in\{1,\dots,n\}$, with the convergence absolute because $\Phi^{i,j}_{\overline{\varepsilon}}(t)\geqslant0$. Hence it follows from the \emph{Weierstrass M-test for uniform convergence of an integral} (see \cite[p. $417$, Proposition~$2$]{Zorich}) that for each $i$, $j\in\{1,\dots,n\}$, the $\varepsilon$-family of improper integrals 
	\[\int_{-\infty}^{+\infty}\Phi^{i,j}_\varepsilon(t)\;dt<\infty,\]
converges absolutely for each $\varepsilon\in[0,\frac{1}{3}\overline{\varepsilon}]$ and uniformly in $[0,\frac{1}{3}\overline{\varepsilon}]$.

Moreover from Proposition~\ref{P1} and Remark~\ref{Remark4} we have that for each closed bounded interval $[a,b]\subset\mathbb{R}$ it holds 
	\[\lim\limits_{\varepsilon\to0}\Phi^{i,j}_\varepsilon(t)=\Phi^{i,j}_0(t),\]
\emph{uniformly} in $t\in[a,b]$ and $\varepsilon\in[0,\frac{1}{3}\overline{\varepsilon}]$, respectively. Thus it follows from \cite[p. $420$, Proposition~$4$]{Zorich} that
	\[\lim\limits_{\varepsilon\to0}\int_{-\infty}^{+\infty}\Phi^{i,j}_\varepsilon(t)\;dt=\int_{-\infty}^{+\infty}\Phi^{i,j}_0(t)\;dt,\]
for every $i$, $j\in\{1,\dots,n\}$. Therefore \eqref{45} holds and the claim is proved.

From definition of the bump-functions $\varphi_j$ we know that if $i\neq j$, then $\varphi_j(\gamma_i(t))\equiv0$. Hence, from \eqref{6} and \eqref{11} we obtain that
\begin{equation}\label{7}
	\frac{\partial d_{i,0}}{\partial\mu_j}(0)=0,
\end{equation}
for every $i$, $j\in\{1,\dots,n\}$, with $i\neq j$. Similarly, if $i=j$, then it follows from \eqref{11} that
\begin{equation}\label{8}
	\frac{\partial d_{i,0}}{\partial\mu_i}(0)>0,
\end{equation}
for every $i\in\{1,\dots,n\}$.

We now deal with the bifurcation of the limit cycles. The proof will be by induction on $n$. First, observe that if $n=1,$ $\mu=\mu_1$ and $\Delta(\Gamma^1,\tau)=1$ (i.e. if $R_1=r_1\neq1$), then by using for instance Andronov et al \cite[$\mathsection29$]{And1971}  we get that $X_{\mu,\varepsilon}$ has a limit cycle near $\Gamma^1$  if, and only if, $(r_1-1)\mu\lesssim 0$. Therefore, from now on assume $n\geqslant2$. Let $\mu=(\mu_1,\dots,\mu_n)$ and suppose that the theorem holds for $n-1$. Assume for now that $\Delta(\Gamma^n,\tau)=n$. That is, assume that $(R_i-1)(R_{i-1}-1)<0$ for every $i\in\{2,\dots,n\}$ and that $R_1\neq1$. For definiteness, assume also that $R_n>1$ and $R_{n-1}<1$. In special, observe that $r_n>1$. Since $R_n>1$, from Cherkas \cite{Cherkas} we know that $\Gamma^n$ is stable. Moreover, for each $\varepsilon\in[0,\overline{\varepsilon}]$, it follows from Proposition~\ref{P2} and Corollary~\ref{Coro1} that $d_{n-1,\varepsilon}^{(1)}\colon\Lambda\to\mathbb{R}$ is a well defined function of class $C^1$ such that,
\begin{equation}\label{43}
	\frac{\partial d_{n-1,\varepsilon}^{(1)}}{\partial\mu_j}(0)=\frac{\partial d_{n-1,\varepsilon}}{\partial\mu_j}(0),
\end{equation}
for every $j\in\{1,\dots,n\}$. For each $\varepsilon\in[0,\overline\varepsilon]$, let $F_{\varepsilon}\colon\Lambda\to\mathbb{R}^{n-1}$ be given by
	\[F_{\varepsilon}(\mu)=\left(d_{1,\varepsilon}(\mu),\dots,d_{n-2,\varepsilon}(\mu),d_{n-1,\varepsilon}^{(1)}(\mu)\right),\]
and consider its $(n-1)\times n$ Jacobian matrix at $\mu=0$,
	\[DF_\varepsilon(0)=\left(\begin{array}{ccccc} 
		\displaystyle \frac{\partial d_{1,\varepsilon}}{\partial\mu_1}(0) & \displaystyle \frac{\partial d_{1,\varepsilon}}{\partial\mu_2}(0) & \dots & \displaystyle \frac{\partial d_{1,\varepsilon}}{\partial\mu_{n-1}}(0) & \displaystyle \frac{\partial d_{1,\varepsilon}}{\partial\mu_n}(0) \vspace{0.2cm} \\
		\displaystyle \frac{\partial d_{2,\varepsilon}}{\partial\mu_1}(0) & \displaystyle \frac{\partial d_{2,\varepsilon}}{\partial\mu_2}(0) & \dots & \displaystyle \frac{\partial d_{2,\varepsilon}}{\partial\mu_{n-1}}(0) & \displaystyle \frac{\partial d_{2,\varepsilon}}{\partial\mu_n}(0) \vspace{0.2cm} \\
		\vdots & \vdots & \ddots & \vdots & \vdots \vspace{0.2cm} \\
		\displaystyle\frac{\partial d_{n-1,\varepsilon}^{(1)}}{\partial\mu_1}(0) & \displaystyle \frac{\partial d_{n-1,\varepsilon}^{(1)}}{\partial\mu_2}(0) & \dots & \displaystyle \frac{\partial d_{n-1,\varepsilon}^{(1)}}{\partial\mu_{n-1}}(0) & \displaystyle \frac{\partial d_{n-1,\varepsilon}^{(1)}}{\partial\mu_n}(0)
	\end{array}\right).\]
Let $A_{\varepsilon}$ be the $(n-1)\times(n-1)$ submatrix of $DF_{\varepsilon}(0)$ given by its first $n-1$ columns. It follows from \eqref{7}, \eqref{8} and \eqref{43} that $\det A_0>0$. Hence, by using \eqref{45} and from the continuity of the determinant we know that $\det A_{\varepsilon}>0$ for $\varepsilon\geqslant0$ small enough. Therefore, if we fix $\varepsilon_0\geqslant0$ small enough, we get from the Implicit Function Theorem that there are unique $C^1$ functions $\mu_i^*=\mu_i^*(\mu_n)$, $i\in\{1,\dots,n-1\}$, with $\mu_i^*(0)=0$ and such that
\begin{equation}\label{48}
	F_{\varepsilon_0}(\mu_1^*(\mu_n),\dots,\mu_{n-1}^*(\mu_n),\mu_n)=0,
\end{equation}
for $|\mu_n|$ small. Moreover, it follows from \eqref{45} and \eqref{8} that,
	\[\frac{\partial d_{n,\varepsilon_0}}{\partial \mu_n}(0)>0.\]
Hence, $d_{n,\varepsilon_0}(\mu)\neq0$ if $\mu_n\neq0$. Therefore, from \eqref{48} we know that for $|\mu_n|\neq0$ small enough  and $\mu_i=\mu_i^*(\mu_n)$,  $X_{\mu,\varepsilon_0}$ has a polycycle $\Gamma^{n-1}=\Gamma^{n-1}(\mu_n)$ formed by $n-1$ hyperbolic saddles $p_1(\mu_n),\dots,p_{n-1}(\mu_n)$, and $n-1$ heteroclinic connections $L_i^*=L_i^*(\mu_n)$. It follows from the Implicit Function Theorem that $p_i(\mu_n)\to p_i$ as $\mu_n\to0$. In addition, from the continuous dependence with respect to initial conditions~\cite[Theorem $8$]{And1971} and the local Center-Stable Manifold Theorem~\cite[Theorem~$1$]{Kelley} we get that the closure $\overline{L_i}$ of each regular orbit $L_i$ of $\Gamma^{n-1}(\mu_n)$ (i.e. the regular orbit together with the two singularities given by its $\alpha$ and $\omega$-limits) converges to the closure of the regular orbits of $\Gamma^n$, in relation to the Hausdorff distance, as $\mu_n\to0$. More precisely, for every $\varepsilon>0$ there is $\delta>0$ such that if $|\mu_n|<\delta$, then the following statements hold.
\begin{enumerate}[label=(\alph*)]
	\item $d_H(\Gamma^{n-1}(\mu_n),\Gamma^n)<\varepsilon$.
	\item $d_H(\overline{L_{n-1}^*(\mu_n)},\overline{L_n\cup L_{n-1}})<\varepsilon$.
	\item $d_H(\overline{L_i^*(\mu_n)},\overline{L_i})<\varepsilon$, for each $i\in\{1,\dots,n-2\}$.
\end{enumerate}
see Figures~\ref{Fig6} and \ref{Fig4}.
\begin{figure}[ht]
	\begin{center}
		\begin{minipage}{8.5cm}
			\begin{center} 
				\begin{overpic}[width=8cm]{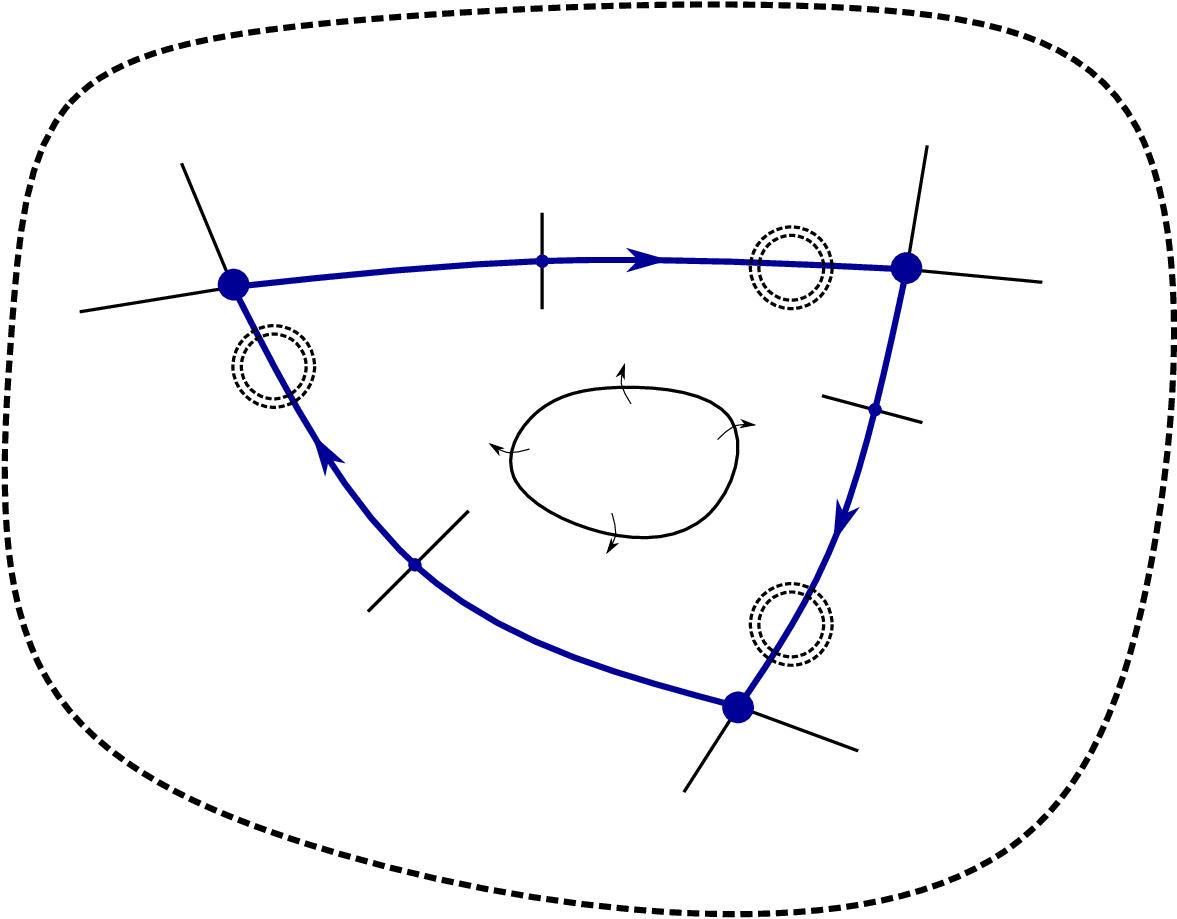} 
					\put(80,57){$p_1$}
					\put(12,55){$p_2$}
					\put(62,13){$p_3$}		
					\put(95,72){$B$}
					\put(56.5,35.5){$S$}		
				\end{overpic}
				
				Before the perturbation.
			\end{center}
		\end{minipage}
		\begin{minipage}{8.5cm}
			\begin{center} 
				\begin{overpic}[width=8cm]{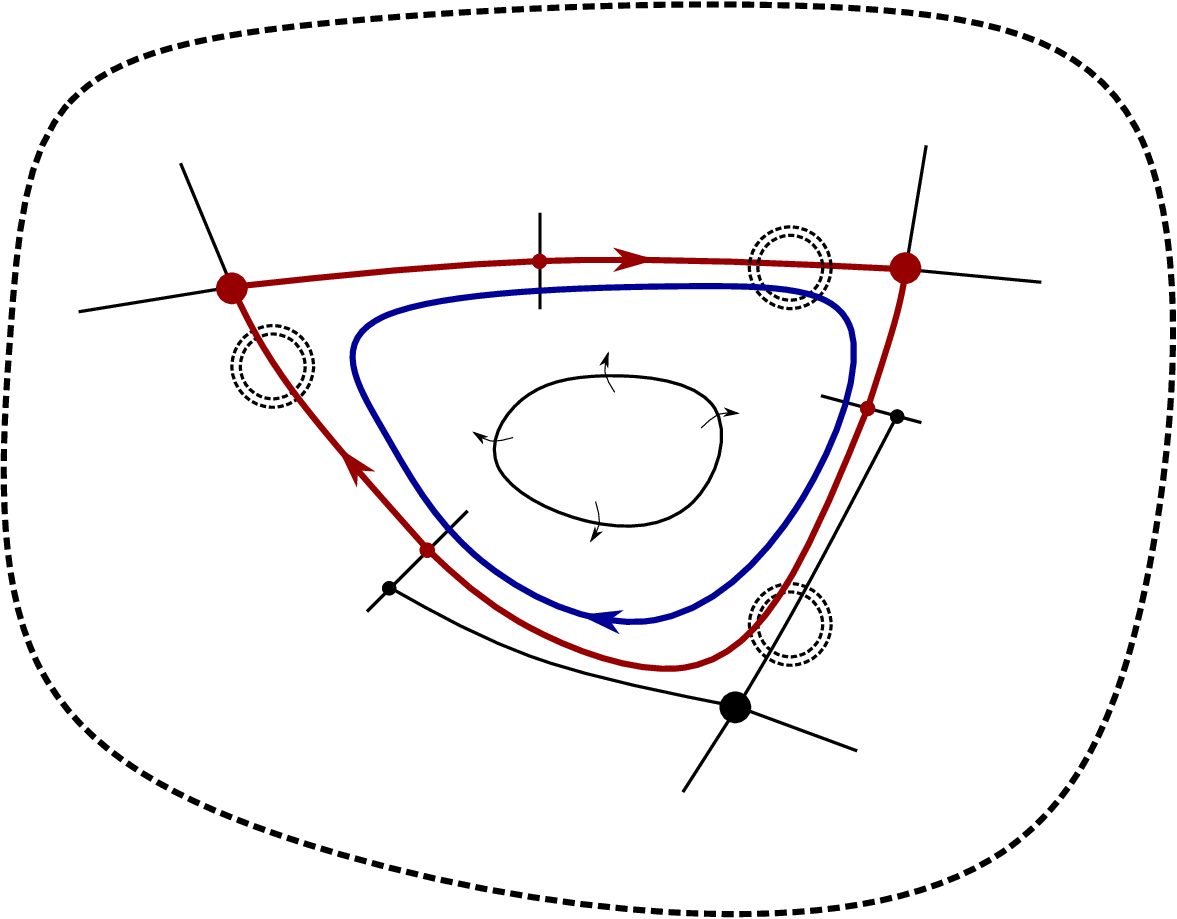} 
					\put(80,57){$p_1$}
					\put(12,55){$p_2$}
					\put(62,13){$p_3$}
					\put(95,72){$B$}
					\put(56,36){$S$}	
				\end{overpic}
				
				After the perturbation.
			\end{center}
		\end{minipage}
	\end{center}

$\;$

	\begin{center}
		\begin{minipage}{8.5cm}
			\begin{center} 
				\begin{overpic}[height=6cm]{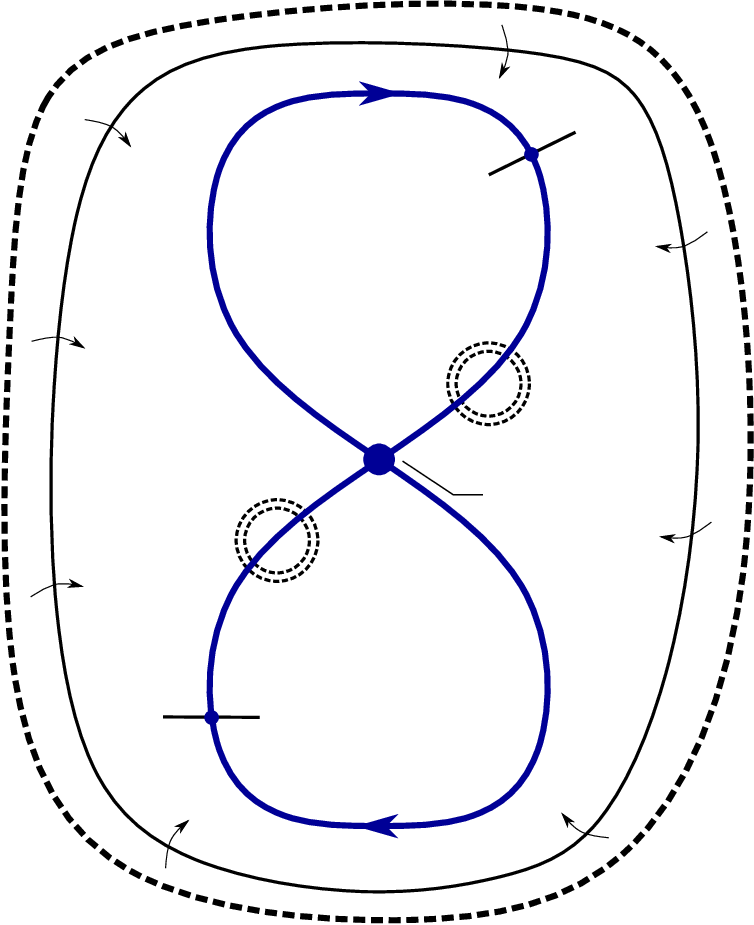} 
					\put(53,45){$p_1=p_2$}
					\put(7,50){$S$}
					\put(72,95){$B$}		
				\end{overpic}
				
				Before the perturbation.
			\end{center}
		\end{minipage}
		\begin{minipage}{8.5cm}
			\begin{center} 
				\begin{overpic}[height=6cm]{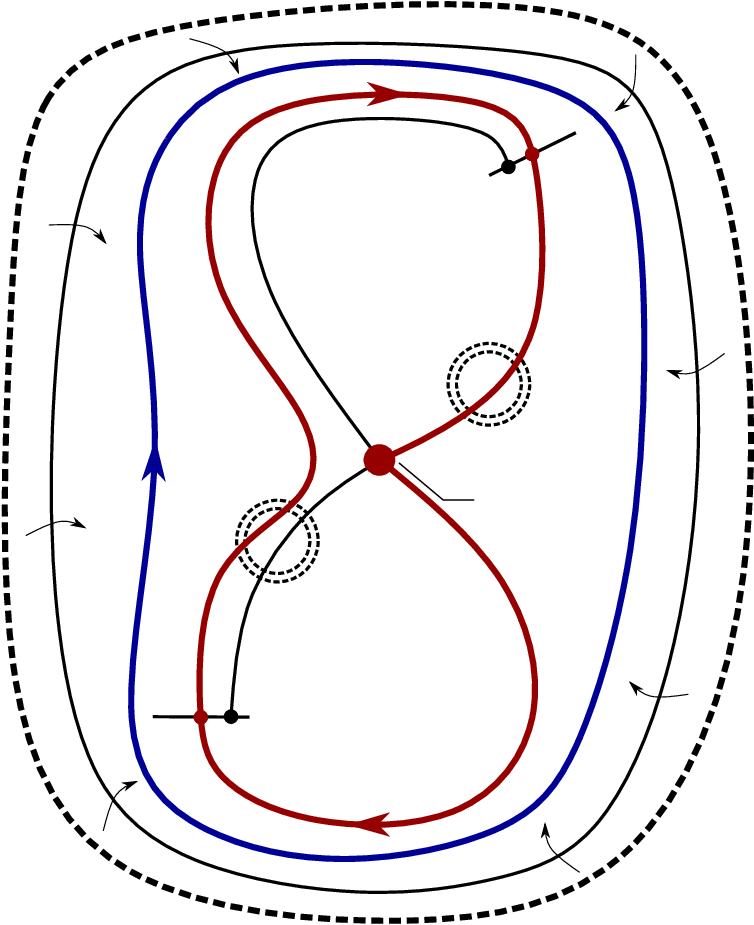} 
					\put(52.5,44.5){$p_1$}
					\put(7,50){$S$}
					\put(72,95){$B$}	
				\end{overpic}
				
				After the perturbation.
			\end{center}
		\end{minipage}
	\end{center}
	\caption{Two llustrations of the bifurcation process. Since $R_n>1$ and $R_{n-1}<1$, it follows that $\Gamma^n$ is stable and $\Gamma^{n-1}$ is unstable.}\label{Fig4}
\end{figure} 
For each $j\in\{1,\dots,n-1\}$, let
	\[R_j^*(\mu_n)=\prod_{i=1}^{j}r_i|_{\mu_i=\mu_i^*(\mu_n),\;i\in\{1,\dots,n-1\}},\]
and observe that $R_{n-1}^*<1$ and $(R_i^*-1)(R_{i-1}^*-1)<0$, $i\in\{2,\dots,n-1\}$, provided $|\mu_n|>0$ is small enough. Hence, $\Gamma^{n-1}$ is unstable if $|\mu_n|\neq0$. Let $S$ be the curve given by Proposition~\ref{P3} and let $\Omega_{\mu_n}$ be the open region bounded by $S$ and $\Gamma^{n-1}(\mu_n)$. Since $\Omega_{\mu_n}$ is positive invariant by the flow of $X_{\mu,\varepsilon_0}$ and $\Gamma^{n-1}$ is unstable, it follows from the Poincar\'e-Bendixson Theorem that there is at least one periodic orbit $C_n(\mu_n)$ in $\Omega_{\mu_n}$ that is not unstable. 

If $X$ is not polynomial then we are outside the analytic framework and thus we might have the bifurcation of infinitely many periodic orbits (see Remark~\ref{Remark6}). In particular, $C_n(\mu_n)$ may not be isolated. In this case we can apply Proposition~\ref{P4}, with the compact support of $\Phi$ small enough such that it does not intersect a neighborhood of $\Gamma^{n-1}$ (and thus does not perturb it), and hence obtain a close enough perturbed vector field that has $C_n(\mu_n)$ as a stable hyperbolic limit cycle. 
	
On the other hand, we may have the bifurcation of at most a finite amount of periodic orbits. In this case every periodic orbit is isolated and thus $C_n(\mu_n)$ is a limit cycle.

If $X$ is polynomial then the perturbation is also polynomial and in particular analytic. This in addition with the fact that $\Gamma^{n-1}$ is unstable (and thus cannot be accumulated by periodic orbits) and the fact that for analytic vector fields all limit cycles are isolated and with finite multiplicity, ensures the bifurcation of at most a finite number of periodic orbits. In particular, $C_n(\mu_n)$ is a limit cycle.

Either in the smooth or polynomial case, we claim that if at most a finite amount of periodic orbits bifurcate, then we can choose $C_n(\mu_n)$ to be stable limit cycle (but not necessarily hyperbolic). Indeed, if $C_n(\mu_n)$ is the unique limit cycle that bifurcates from $\Gamma^n$, then it is clear that it is stable. Suppose therefore that there are the bifurcation of $k$ nested limit cycles $\gamma_1(\mu_n),\dots,\gamma_k(\mu_n)$, with $\gamma_{j-1}$ in the bounded region limited by $\gamma_j$.  Since $\gamma_k$ is the outermost limit cycle, it follows that $\gamma_k$ is stable from outside. Therefore if $\gamma_k$ is not stable, then it is unstable from the inside and thus $\gamma_{k-1}$ is stable from the outside. Similarly, if $\gamma_{k-1}$ is not stable then it is unstable from the inside. Therefore if none of $\gamma_2,\dots,\gamma_k$ are stable, then $\gamma_1$ must be stable from the outside. However since $\gamma_1$ is the innermost limit cycle, it also follows that it is stable from the inside and thus $\gamma_1$ is stable. This proves the claim. 

In particular either in the smooth or polynomial case, observe that $C_n(\mu_n)$ has odd multiplicity and thus its existence persist for small perturbations. 

Hence, if we fix $\mu_n=\mu_n^*$ small enough and let $\mu^*=(\mu_1^*(\mu_n^*),\dots,\mu_{n-1}^*(\mu_n^*),\mu_n^*)$, then it follows by induction that there is an arbitrarily small perturbation of $X_{\mu^*,\varepsilon_0}$ bifurcating at least $n-1$ limit cycles from $\Gamma^{n-1}(\mu_n^*)$. Since $C_n(\mu_n^*)$ persists for small perturbations, we have the bifurcation of at least $n$ limit cycles from $\Gamma^n$. This proves the theorem for the case $\Delta(\Gamma^n,\tau)=n$.

We now study the general case. First observe that to expel $p_n$ it is only necessary to have $r_n\neq1$, regardless of having $r_k=1$ for some other $k\in\{1,\dots,n-1\}$. This can be seen by the definition of the map $F_{\varepsilon}\colon\Lambda\to\mathbb{R}^{n-1}$ given by,
	\[F_{\varepsilon}(\mu)=\left(d_{1,\varepsilon}(\mu),\dots,d_{n-2,\varepsilon}(\mu),d_{n-1,\varepsilon}^{(1)}(\mu)\right).\]
More precisely observe that components $d_{1,\varepsilon},\dots,d_{n-2,\varepsilon}$ are always smooth, while the last component $d_{n-1,\varepsilon}^{(1)}$ is smooth because $r_n\neq1$. 

If $\Delta(\Gamma^n)=n$, then there is a permutation of the indexes $\sigma$ such that $\Delta(\Gamma^n,\sigma)=n$. In particular it is not hard to see that $r_{\sigma(i)}\neq1$ for every $i\in\{1,\dots,n\}$ and thus we can expel the singularities $p_{\sigma(n)},\dots,p_{\sigma(1)}$ one at a each step. Moreover it follows from the definition of $\Delta(\Gamma^n,\sigma)=n$ that at each step the stability of the polycycle reverses and then we have the bifurcation of at least one limit cycle.

If $\Delta(\Gamma^n)=n-1$ then there is a permutation of the indexes $\sigma$ such that $\Delta(\Gamma^n,\sigma)=n-1$. To simplify the notation we shall assume that $\sigma$ is the trivial permutation $\tau$. From $\Delta(\Gamma^n,\tau)=n-1$ we have that there is a unique $i_0\in\{1,\dots,n\}$ such that
	\[(R_{i_0-1}-1)(R_{i_0}-1)\geqslant0, \quad (R_{i-1}-1)(R_i-1)<0, \quad i\in\{1,\dots,i_0-1,i_0+1,\dots,n\},\]
where $R_i=\prod_{j=1}^{i}r_j$, $i\in\{1,\dots,n\}$, and $R_0=R_1^{-1}$. Observe that if $(R_{i-1}-1)(R_i-1)<0$, then $r_i\neq1$.

If $i_0<n$ then we can expel the singularities $p_n,p_{n-1},\dots,p_{i_0+1}$ one at a time and obtain a limit cycle at each step. In particular, we now have $n-{i_0}$ limit cycles and a polycycle with hyperbolic saddles $p_1,\dots,p_{i_0-1},p_{i_0}$ such that
	\[(R_{i_0-1}-1)(R_{i_0}-1)\geqslant0, \quad (R_{i-1}-1)(R_i-1)<0, \quad i\in\{1,\dots,i_0-1\}.\]
If $r_{i_0}\neq1$ then we just expel $p_{i_0}$ (resulting in no limit cycles in this particular step) and thus the following steps are now free to proceed normally. 

If $r_{i_0}=1$ then it plays no role in the alternation of the signs of $R_i-1$ and thus we can take a new indexation given by $p_{i_0}\mapsto p_1$ and $p_i\mapsto p_{i+1}$ for $i\in\{1,\dots,i_0-1\}$. We now have a polycycle such that
	\[(R_0-1)(R_1-1)=0, \quad (R_1-1)(R_2-1)=0, \quad (R_{i-1}-1)(R_i-1)<0, \quad i\in\{3,\dots,i_0\},\]
with the two equations on the left-hand side due to $R_1=r_{i_0}=1$ and $R_0=R_1^{-1}=1$.

Hence we can expel the hyperbolic saddles $p_{i_0},\dots,p_3$, obtaining $i_0-2$ more limit cycles, which adds up to $n-2$ with the previous $n-i_0$ that we had already bifurcated. We now have a polycycle $\Gamma^2$ with two hyperbolic saddles $p_1$ and $p_2$ such that $r_1=r_{i_0}=1$ and $r_2\neq1$ and we must obtain one more limit cycle. To do this, observe that $\Gamma^2$ has a well defined stability because $R_2=r_2\neq1$. Let $S$ be the curve given by Proposition~\ref{P3}. Instead of expelling $p_2$, we now use the displacement maps $d_{1,\varepsilon}$ and $d_{2,\varepsilon}$ to break $\Gamma^2$ in such a way that in addition with the curve $S$, it creates an invariant region $\Omega$ from which the Poincar\'e-Bendixson Theorem ensures the existence of at least one limit cycle. 

The case $\Delta(\Gamma^n)=n-k$ for some $k\in\{1,\dots,n-2\}$ follows similarly. The only difference is that at the end to bifurcate the last limit cycle we may have a polycycle $\Gamma^{k_0+1}$, for some $k_0\in\{1,\dots,k\}$, such that all its hyperbolic saddles $p_1,\dots,p_{k_0},p_{k_0+1}$, except one, have hyperbolicity number $r_j=1$. In particular $\Gamma^{k_0+1}$ has a well defined stability and thus we can apply Proposition~\ref{P3}. At this point we can use the displacement functions $d_{1,\varepsilon},\dots,d_{k_0+1,\varepsilon}$ to break all the heteroclinic connections of $\Gamma^{k_0+1}$ in such a way that we can apply the Poincar\'e-Bendixson Theorem to bifurcate at least one more limit cycle. See Figure~\ref{Fig8}.
\begin{figure}[h]		
	\begin{center}
		\begin{minipage}{7cm}
			\begin{center} 
				\begin{overpic}[width=4cm]{Fig10.eps} 
					\put(78,85){$p_1$}
					\put(14,85){$p_2$}
					\put(-10,43){$p_3$}
					\put(14,0){$p_4$}
					\put(78,0){$p_5$}	
					\put(101,43){$p_6$}
					\put(65,35){$S$}
				\end{overpic}
			
				$\;$
			
				Unperturbed.
								
			\end{center}
		\end{minipage}
		\begin{minipage}{7cm}
			\begin{center} 
				\begin{overpic}[width=4cm]{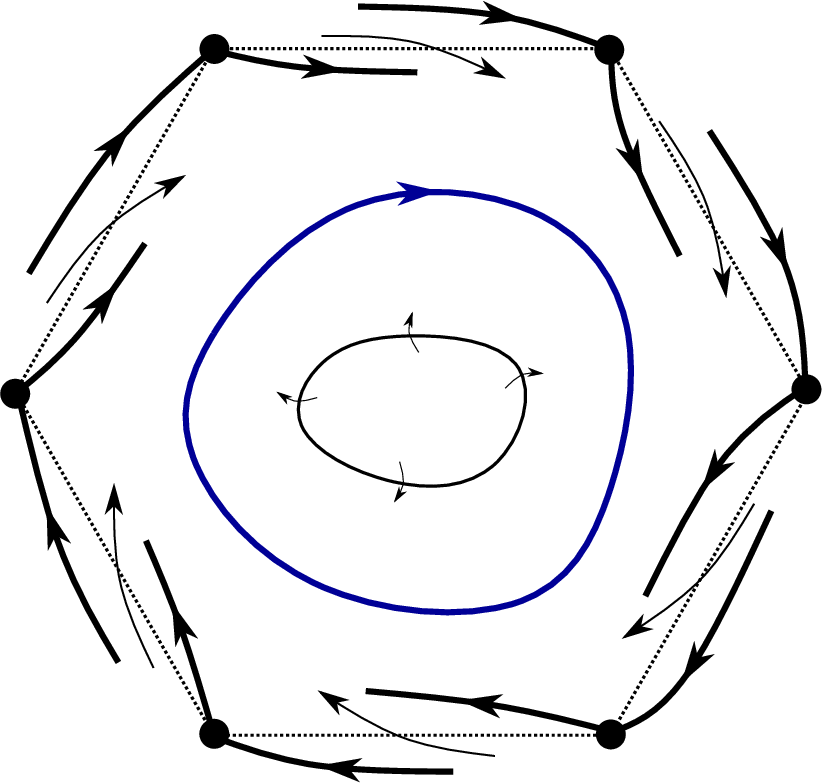} 
					\put(78,90){$p_1$}
					\put(14,90){$p_2$}
					\put(-9,45){$p_3$}
					\put(16,0){$p_4$}
					\put(76,0){$p_5$}	
					\put(101,45){$p_6$}
					\put(62,35){$S$}
				\end{overpic}
			
				$\;$
				
				Perturbed.
				
			\end{center}
		\end{minipage}
	\end{center}
	\caption{Illustration of the bifurcation process with $n=6$, $k_0=5$ and such that there exists a unique $i_0\in\{1,\dots,6\}$ such that $r_{i_0}>1$ and $r_i=1$ for $i\neq i_0$. In particular observe that $\Delta(\Gamma^6)=1$. Blue means stable. Colors available in the online version.}\label{Fig8}
\end{figure} 
	
Finally, observe that if $X$ is polynomial and $\varepsilon_0>0$ is small enough, then $X_{\mu^*,\varepsilon_0}$ is also polynomial. On the other hand, if $X$ is smooth, then its approximations constructed in the proof are smooth as well. {\hfill$\square$}

\begin{remark}\label{Remark6}
	At the proof of Theorem~\ref{Main1} we observe that the case in which $X$ is smooth and infinitely many periodic orbits bifurcate from it is an exceptional case. More precisely, Mourtada~\cite[Theorem~$3$]{Mourtada5} proved that even in the smooth case, generically speaking at most a finite amount of periodic orbits bifurcate from a given hyperbolic polycycle. For more details, see Section~\ref{Sec7}.
\end{remark}

\section{The inverse problem and a concrete example}\label{Sec6}

We start  this section by considering an inverse problem. More concretely  the problem of constructing a polycycle $\Gamma^n$ from a given set $\{r_1,\dots,r_n\}$ of desired hyperbolicity ratios. In particular, we prove that every possibility is realizable by a polynomial vector field of degree at most $n.$

\begin{proposition}\label{Main2}
	Given $n\geqslant 3$, let $r_1,\dots,r_n\in\mathbb{R}$ be positive real numbers. Then there is a planar polynomial vector field $X$ of degree at most $n$ with a polycycle $\Gamma^n$ composed by $n$ distinct hyperbolic saddles $p_1,\dots,p_n$ such that $r_i$ is the hyperbolicity ratio of $p_i$, $i\in\{1,\dots,n\}$.
\end{proposition}

\begin{proof} For each $i\in\{1,\dots,n\}$ let $\xi_i=\cos(2i\pi/n)+i\sin(2i\pi/n)$ be the roots of unity of order $n$. For each $\xi_i\in\mathbb{C}$ we associate the point $p_i\in\mathbb{R}^2$ given by $p_i=(\cos(2i\pi/n),\sin(2i\pi/n))$. It is well known that $\xi_1,\dots,\xi_n\in\mathbb{C}$ divides the unit circle equally and thus it can be seen as the vertices of a regular polygon of $n$ edges. Hence, the points $p_1,\dots,p_n\in\mathbb{R}^2$ can also be seen as the vertices of a regular polygon $\Gamma^n\subset\mathbb{R}^2$ of $n$ edges. Let $l_1,\dots,l_n\subset\mathbb{R}^2$ be the $n$ straight lines such that $l_i\cap l_{i-1}=\{p_i\}$, $i\in\{1,\dots,n\}$, with $l_{0}=l_n,$ see Figure~\ref{Fig5}.
\begin{figure}[ht]
	\begin{center}
		\begin{minipage}{7cm}
			\begin{center} 
				\begin{overpic}[height=4cm]{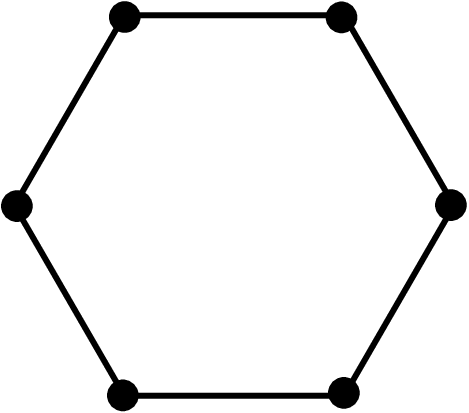} 
					\put(101,43){$p_6$}
					\put(78,85){$p_1$}
					\put(14,85){$p_2$}
					\put(-9,43){$p_3$}
					\put(14,0){$p_4$}
					\put(78,0){$p_5$}
					
					\put(48,75){$l_1$}
					\put(20,60){$l_2$}
					\put(20,25){$l_3$}
					\put(48,8){$l_4$}	
					\put(75,25){$l_5$}
					\put(75,60){$l_6$}
				\end{overpic}
			
			$n=6$.
			\end{center}
		\end{minipage}
		\begin{minipage}{7cm}
			\begin{center} 
				\begin{overpic}[height=4.1cm]{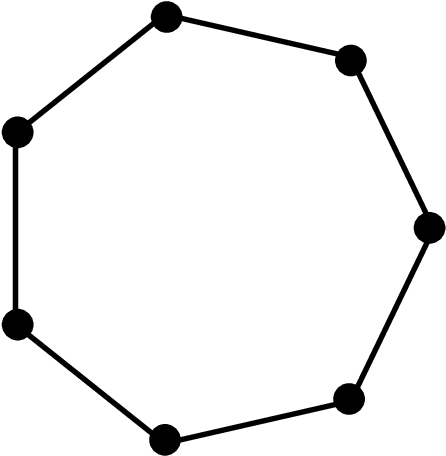}
					\put(101,49){$p_7$}
					\put(81,91){$p_1$}
					\put(22,98){$p_2$}
					\put(-10,70){$p_3$}
					\put(-10,28){$p_4$}
					\put(22,0){$p_5$}
					\put(81,8){$p_6$}
					
					\put(53,82){$l_1$}
					\put(22,77){$l_2$}
					\put(8,47){$l_3$}
					\put(22,19){$l_4$}
					\put(53,12){$l_5$}
					\put(75,32){$l_6$}
					\put(75,65){$l_7$}
				\end{overpic}
			
			$n=7$.
			\end{center}
		\end{minipage}
	\end{center}
	\caption{Illustration of $\Gamma^n$ with $n=6$ and $n=7$.}\label{Fig5}
\end{figure} 
Let $\alpha_i$, $\beta_i$, $d_i\in\mathbb{R}$ be such that $l_i$ is given by $\alpha_ix_1+\beta_ix_2-d_i=0$ and write $l_i(x)=\alpha_ix_1+\beta_ix_2-d_i$. Let also $X=(P,Q)$ be the planar polynomial system of degree $n$ given by	
\begin{equation}\label{25}
	P(x)=-\sum_{i=1}^{n}\left[\beta_iA_i(x)\prod_{j\neq i}l_j(x)\right], \quad  Q(x)=\sum_{i=1}^{n}\left[\alpha_iA_i(x)\prod_{j\neq i}l_j(x)\right].
\end{equation}
with $\deg A_i=1$, $i\in\{1,\dots,n\}$. We claim that each $l_i$ is an invariant straight line of $X$. Indeed, let $w\in l_s$ and observe that
\begin{equation}\label{26}
	P(w)=-\beta_sA_s(w)\prod_{j\neq s}l_s(w), \quad Q(w)=\alpha_sA_s(w)\prod_{j\neq s}l_s(w).
\end{equation}
The claim now follows from the fact that $\left<X(w),(\alpha_s,\beta_s)\right>=0$, where recall $\left<\cdot,\cdot\right>$ denotes the standard inner product of $\mathbb{R}^2$.  We now study the Jacobian matrix of $X$ at $p_s$, $s\in\{1,\dots,n\}$. It follows from \eqref{25} that,
\begin{equation}\label{27}
	\frac{\partial P}{\partial x_1}=-\sum_{i=1}^{n}\left[\beta_i\frac{\partial A_i}{\partial x_1}\prod_{j\neq i}l_j+\beta_iA_i\sum_{k\neq i}\left(\alpha_k\prod_{j\neq i,k}l_j\right)\right].
\end{equation}
For each $s\in\{1,\dots,n\}$ let,
\begin{equation}\label{41}
	M(s)=\prod_{\substack{j\neq s \\ j\neq s-1}}l_j(p_s).
\end{equation}
Since $l_s(p_s)=l_{s-1}(p_s)=0$, from \eqref{27} we obtain that
\begin{equation}\label{28}
	\frac{\partial P}{\partial x_1}(p_s)=-M(s)\bigl(\alpha_{s-1}\beta_sA_s(p_s)+\alpha_s\beta_{s-1}A_{s-1}(p_s)\bigr).
\end{equation}
Similarly, 
\begin{equation}\label{29}
	\begin{array}{l}
		\displaystyle \frac{\partial P}{\partial x_2}(p_s)=-\beta_s\beta_{s-1}M(s)\bigl(A_s(p_s)+A_{s-1}(p_s)\bigr), \vspace{0.2cm} \\
		\displaystyle \frac{\partial Q}{\partial x_1}(p_s)=\alpha_s\alpha_{s-1}M(s)\bigl(A_s(p_s)+A_{s-1}(p_s)\bigr), \vspace{0.2cm} \\
		\displaystyle \frac{\partial Q}{\partial x_2}(p_s)=M(s)\bigl(\alpha_s\beta_{s-1}A_s(p_s)+\alpha_{s-1}\beta_sA_{s-1}(p_s)\bigr).
	\end{array}
\end{equation}
Hence, from \eqref{28} and \eqref{29} the determinant of the Jacobian matrix of $X$ at $p_s$ is
\begin{equation}\label{30}
	\det DX(p_s)=-M(s)^2(\alpha_s\beta_{s-1}-\alpha_{s-1}\beta_s)^2A_s(p_s)A_{s-1}(p_s).
\end{equation}
Since $p_s\in l_i$ if, and only if $i\in\{s,s-1\}$, it follows that $M(s)\neq0$. Moreover, observe that
\begin{equation}\label{39}
	\alpha_s\beta_{s-1}-\alpha_{s-1}\beta_s=\det\left(\begin{array}{cc} \alpha_s & \alpha_{s-1} \\ \beta_s & \beta_{s-1} \end{array}\right).
\end{equation}
Since $l_s$ and $l_{s-1}$ are never parallel, we know that \eqref{39} never vanishes. Therefore, it follows from \eqref{30} that $p_s$ is a hyperbolic saddle if, and only if, $A_s(p_s)A_{s-1}(p_s)>0$. Moreover, its eigenvalues are given by
\begin{equation}\label{31}
	\mu_s=-(\alpha_s\beta_{s-1}-\alpha_{s-1}\beta_s)M(s)A_{s}(p_s), \quad \nu_s=(\alpha_s\beta_{s-1}-\alpha_{s-1}\beta_s)M(s)A_{s-1}(p_s).
\end{equation}
Given $s\in\{1,\dots,n\}$, let $w\in l_s$ be in the segment between $p_{s+1}$ and $p_s$. For $\Gamma^n$ to be a polycycle, is necessary that $w$ is not a singularity of $X$. It follows from \eqref{26} that $w$ is a singularity if, and only if $A_s(w)=0$. Hence, we conclude that $\Gamma^n$ is a polycycle composed by $n$ hyperbolic saddles $p_1,\dots,p_n$ if, and only if, $A_s(p_s)A_{s-1}(p_s)>0$ and $A_s(w)\neq 0$, for every $w\in l_s$ in the segment between $p_{s+1}$ and $p_s$, $s\in\{1,\dots,n\}$ and $p_{n+1}=p_1$. We now study the hyperbolicity ratio of $p_s$. Observe that we can choose $\alpha_s$ and $\beta_s$ such that $v_s=(\alpha_s,\beta_s)$ is unitary. Hence, \eqref{39} is the sine of the angle between $\ell_{s-1}$ and $\ell_s$. Since $\Gamma^n$ is a regular polygon, it follows that there is $\theta_n\in(0,\pi)$ such that $\alpha_{s-1}\beta_s-\alpha_s\beta_{s-1}=\sin\theta_n$, for every $s\in\{1,\dots,n\}$. Observe that we can choose $v_s$ to points towards the bounded region of $\Gamma^n$, for every $s\in\{1,\dots,n\}$. 
 
Moreover, observe that $l_j(p_s)$ is the distance with sign between $p_s$ and $l_j$. Since $\Gamma^n$ is a regular polygon, $v_s$ is unitary and points towards the bounded region of $\Gamma^n$, we get from \eqref{41} that $M(s)=M_n>0$, for every $s\in\{1,\dots,n\}$. Therefore from \eqref{31} we obtain that
\begin{equation}\label{40}
	\mu_s=-\sin\theta_nM_nA_s(p_s), \quad \nu_s=\sin\theta_nM_nA_{s-1}(p_s).
\end{equation}
Thus, if we choose $A_1,\dots, A_n$ such that $A_s(p_s)>0$ and $A_{s-1}(p_s)>0$, then we conclude that the hyperboli\-city ratio of $p_s$ is given by,
\begin{equation}\label{32}
	\frac{|\mu_s|}{\nu_s}=\frac{A_{s}(p_s)}{A_{s-1}(p_s)}, 
\end{equation}
for $s\in\{1,\dots,n\}$. Given $r_1,\dots,r_n\in\mathbb{R}$ positive real numbers, we can choose the polynomial $A_s\colon\mathbb{R}^2\to\mathbb{R}$ of degree one such that
\begin{equation}\label{33}
	A_s(p_s)=r_s, \quad A_{s-1}(p_s)=1,
\end{equation}
for every $s\in\{1,\dots,n\}$. Hence, from \eqref{32} we know that $r_s$ is the hyperbolicity ratio of $p_s$. Moreover, observe that since $A_s(p_s)=r_s>0$, $A_{s}(p_{s+1})=1>0$ and $\deg A_s=1$, it follows that $A_s(w)>0$ for every $w\in l_s$ in the segment between $p_{s+1}$ and $p_s$. Therefore, $X$ has no singularities between $p_{s+1}$ and $p_s$ and thus $\Gamma^n$ is indeed a polycycle. \end{proof}

\begin{remark}
	We observe that the construction presented at Proposition~\ref{Main2} results in a polycycle with the clockwise orientation, see Figure~\ref{Fig6}. If one wants a polycycle with the counter clockwise orientation, then it is sufficient to replace \eqref{33} by
		\[A_s(p_s)=-1, \quad A_{s-1}(p_s)=-r_s.\]
	In particular, the hyperbolicity ratio is now given by $|\nu_s|/\mu_s$.
\end{remark}

We end this section with an example.

\begin{proposition}\label{Main3}
	Set $n\geqslant 3$. Then there is a polynomial vector field  $X$ of degree $n$ with a polycycle $\Gamma^n$ that has cyclicity at least $n$ inside the space of polynomial vector fields of degree $n$, with the coefficients topology.
\end{proposition}

\begin{proof} Given $n\geqslant 3$, let $r_1,\dots,r_n\in\mathbb{R}$ be positive real numbers and consider $R_i=\prod_{j=1}^{i}r_j$. Observe that we can choose $r_1,\dots,r_n$ recursively such that $R_1\neq1$ and $(R_i-1)(R_{i-1}-1)<0$, for every $i\in\{2,\dots,n\}$. Without loss of generality, we can suppose $r_n>1$. For these $r_1,\dots,r_n$, let $X$ be the planar polynomial vector field of degree $n$ given by Proposition~\ref{Main2}. That is, let $X=(P,Q)$ be given by \eqref{25},
\begin{equation*}
	P(x)=-\sum_{i=1}^{n}\left[\beta_iA_i(x)\prod_{j\neq i}l_j(x)\right], \quad  Q(x)=\sum_{i=1}^{n}\left[\alpha_iA_i(x)\prod_{j\neq i}l_j(x)\right],
\end{equation*}
where $l_i(x)=\alpha_ix_1+\beta_ix_2-d_i$ are such that the straight lines $l_i(x)=0$ are invariant and satisfy $l_i\cap l_{i-1}=\{p_i\}$, with $p_i=(\cos(2i\pi/n),\sin(2i\pi/n))$, for $i\in\{1,\dots,n\}$. Moreover,  recall that $\deg A_i=1$ and $A_i(w)>0$ for every $w\in l_i$ in the segment between $p_{i+1}$ and $p_{i}$, $i\in\{1,\dots,n\}$. Without loss of generality we can assume that $\Gamma^n$ has the clockwise orientation. For $s\in\{1,\dots,n\}$ let $H_s\colon\mathbb{R}^2\to\mathbb{R}$ be the polynomial of degree $n-1$ given by,
\begin{equation*}
	H_s(x)=\prod_{j\neq s}l_j(x).
\end{equation*}
Consider now the polynomial $K\colon\mathbb{R}^2\times\mathbb{R}^n\to\mathbb{R}^2$
\begin{equation*}
	K(x,\mu)=\sum_{s=1}^{n}\mu_sH_s(x)Y_s(x),
\end{equation*}
where $\mu=(\mu_1,\dots,\mu_n)\in\mathbb{R}^n$ and $Y_s(x)$ is the constant vector field given by $Y_s(x)\equiv Y_s\equiv(-\alpha_s,-\beta_s)$. Define
	\[X_\mu(x)=X(x)+K(x,\mu).\]
Since $K$ has degree $n-1$ in $x\in\mathbb{R}^2$, it follows that $X_\mu$ is a polynomial vector field of degree $n$. Moreover, clearly $X_\mu\to X$ in the coefficients topology, as $\mu\to0$. Let $\Lambda\subset\mathbb{R}^n$ be a small enough neighborhood of the origin and let $d_i\colon\Lambda\to\mathbb{R}$ be the displacement maps associated to the heteroclinic connections of $\Gamma^n$, $i\in\{1,\dots,n\}$. Let also $d_{n-1}^{(1)}\colon\Lambda\to\mathbb{R}$ be the displacement map given by Proposition~\ref{P2} (recall that $r_n>1$). Notice also that 
\begin{equation}\label{36}
	X(x)\land\frac{\partial K}{\partial \mu_j}(x,0)=(P,Q)\land(-H_j\alpha_j,-H_j\beta_j)=H_j(-P\beta_j+Q\alpha_j).
\end{equation}	
Let $L_i\subset l_i$ be the heteroclinic connection of $\Gamma^n$ from $p_{i+1}$ to $p_i$. Similarly to the proof of Theorem~\ref{Main1}, we now study the sign of \eqref{36} on $L_i$. Let $x_i\in L_i$ and let $\gamma_i(t)$ be the parametrization of $L_i$ given by the solution of $X$, with initial condition $\gamma_i(0)=x_i$. It follows from \eqref{26} that,
\begin{equation}\label{37}
	P(\gamma_i(t))=-\beta_iA_i(\gamma_i(t))H_i(\gamma_i(t)), \quad Q(\gamma_i(t))=\alpha_iA_i(\gamma_i(t))H_i(\gamma_i(t)).
\end{equation}
Replacing \eqref{37} at \eqref{36} and knowing that $(\alpha_i,\beta_i)$ is unitary we obtain, 
\begin{equation}\label{38}
	X(\gamma_i(t))\land\frac{\partial K}{\partial \mu_j}(\gamma_i(t),0)=H_j(\gamma_i(t))H_i(\gamma_i(t))A_i(\gamma_i(t)).
\end{equation}	
Since $R_j\circ\gamma_i=0$ if $i\neq j$, we conclude from \eqref{5} and \eqref{38} that if $i\neq j$, then $\frac{\partial d_i}{\partial \mu_j}(0)=0$.	Moreover, if $i=j$, then it follows from $A_i(\gamma_i(t))>0$ that $\frac{\partial d_i}{\partial \mu_i}(0)>0$, for every $i\in\{1,\dots,n\}$. Since $r_n>1$, it follows from Proposition~\ref{P2} and Corollary~\ref{Coro1} that $d_{n-1}^{(1)}\colon\Lambda\to\mathbb{R}$ is a well defined function of class $C^1$ such that,
\begin{equation*}
	\frac{\partial d_{n-1}^{(1)}}{\partial\mu_j}(0)=\frac{\partial d_{n-1}}{\partial\mu_j}(0),
\end{equation*}
for every $j\in\{1,\dots,n\}$. Then we can define  $F\colon\Lambda\subset\mathbb{R}^n\to\mathbb{R}^{n-1}$ and
	\[F(\mu)=\left(d_1(\mu),\dots,d_{n-2}(\mu),d_{n-1}^{(1)}(\mu)\right),\]
and study its zero set to know the limit cycles and polycycles that persist. At this point the proof can be continued similarly to the one of Theorem~\ref{Main1} with minor changes and we omit all the details. \end{proof}

\section{Final considerations}\label{Sec7}

Let $X$ be a planar $C^\infty$-vector field with a hyperbolic polycycle $\Gamma^n$ with hyperbolic saddles $\{p_1,\dots,p_n\}$, hyperbolicity ratios $r_1,\dots,r_n\in\mathbb{R}_{>0}$ and distinct regular orbits $\{L_1,\dots,L_n\}$, where $p_i$ is the $\omega$-limit of $L_i$. Let also $X_\mu$, with $\mu\in\Lambda$ and $\Lambda\subset\mathbb{R}^n$ a small enough neighborhood of the origin, be a $n$-parameter $C^\infty$-family of vector fields such that $X_0=X$. Let also $d_i\colon\Lambda\to\mathbb{R}$ be the associated displacement map of $L_i$, $i\in\{1,\dots,n\}$.

As anticipated in Remark~\ref{Remark6}, it follows from Mourtada~\cite[Theorem~$3$]{Mourtada5} that generically speaking even in the smooth case the cyclicity of $\Gamma^n$ is finite and depends only on the number $n$ of hyperbolic saddles. More precisely, for each $n\in\mathbb{N}$ there is a finite set of generic algebraic conditions
\begin{equation}\label{46}
	g_{j,n}(r_1,\dots,r_n)\neq0, \quad j\in\{1,\dots,N(n)\},
\end{equation}
with $g_{j,n}$ polynomials of $n$ variables and with integer coefficients; and an integer number $e(n)$ that depends only on $n$, such that for any smooth vector field $X$ with a polycycle $\Gamma^n$, with hyperbolicity ratios satisfying \eqref{46}, and any perturbation family $X_\mu$ of $X$, it holds $\textit{Cycl }(X,X_\mu,\Gamma^n)\leqslant e(n)$. See~\cite[p. $722$]{Mourtada5}.

Among the generic conditions we have those named by Mourtada~\cite[p. $722$]{Mourtada5} as ``$CH$-conditions'' (\emph{Condition Hyperbolique}), given by:
\begin{enumerate}
	\item[{[CH]}] For each subset $J\subset\{1,\dots,n\}$, $\prod_{j\in J}r_j\neq1$.
\end{enumerate}
For $n\leqslant3$ these are the only conditions. For $n\geqslant4$ other conditions appear, see~\cite[p. 723]{Mourtada5}. So far it is known that $e(n)=n$ for $n\leqslant3$ and $e(4)=5$, see~\cite{Mourtada2,Mourtada3,Mourtada4,MourtadaThesis} and the references therein. Explicit upper bounds for $e(n)$ are known for $n\geqslant5$ but they are extremely large and believed to be not sharp. For example, it is known that $e(5)\leqslant65533$, see~\cite{Panazzolo}.

The semi-algebraic conditions \eqref{46} define an open and dense semi-algebraic subset $U$ in $\mathbb{R}^n$ (the space of the hyperbolicity ratios $(r_1,\dots,r_n)$) and for each connected component of $U$ there is a given cyclicity, Roussarie~\cite[Remark~$30$]{Roussarie}. 

Therefore we observe that Theorem~\ref{Main1} provides a lower bound on each one of these connected components. In particular, since $U$ is open and dense, it follows that it contains a $n$-tuple $(r_1,\dots,r_n)$ such that 
	\[(R_i-1)(R_{i-1}-1)<0, \quad \forall i\in\{1,\dots,n\}.\] 
Hence Theorem~\ref{Main1} also provides a new proof for the already known fact~\cite{MourtadaThesis} that $e(n)\geqslant n$ for every $n\in\mathbb{N}$. Moreover, it follows from Propositions~\ref{Main2} and~\ref{Main3} that this lower bound is realizable by \emph{polynomial} vector fields of degree $n$, with the perturbation also polynomial of degree at most $n$, and arbitrarily small in relation to the coefficients topology. 

Let $\ell=l_1$ be the transversal section at the regular orbit $L_1$ where the displacement map $d_1$ takes place, endowed with a coordinate system identifying $\ell$ with $\{t\in\mathbb{R}\colon |t|<\varepsilon\}$, $\varepsilon>0$ small enough, such that $t=0$ is the intersection point $\ell\cap\Gamma^n$ and $0<t<\varepsilon$ is contained in the domain of the first return map associated with $\Gamma^n$. Let also $b_i(\mu)=\sigma_0d_i(\mu)$, $i\in\{1,\dots,n\}$ (where we recall that $\sigma_0\in\{-1,1\}$ depends whether the first return map is defined in the inner or outer region of $\Gamma^n$, see Section~\ref{Sec3}). If the hyperbolicity ratios $r_1,\dots,r_n$ satisfies the generic conditions \eqref{46}, there is a continuous function $\rho\colon\Lambda\to\mathbb{R}$, with $\rho(0)=0$, such that the first return map $\pi\colon(\rho(\mu),\varepsilon)\times\Lambda\to\ell$ is well defined and the solutions of $\pi(t,\mu)=t$ (i.e. the periodic orbits that bifurcate form $\Gamma^n$) are also solutions of
\begin{equation}\label{47}
	\Bigr(\dots\bigr((t^{r_1(\mu)}+b_1(\mu))^{r_2(\mu)}+b_2(\mu)\bigl)^{r_3(\mu)}\dots+b_{n-1}(\mu)\Bigl)^{r_n(\mu)}+b_n(\mu)=\alpha(\mu)t,
\end{equation}
with $\alpha(\mu)>0$ for every $\mu\in\Lambda$. See~\cite{MourtadaThesis} and \cite[Theorem~$1$ and p. $276$]{Mourtada}. In particular, we have that the generic cyclicity $e(n)$ is bounded above by the maximum number $\operatorname{fp}(n)$ of solutions of equation~\eqref{47}. We observe that these number need not to be equal because \eqref{47} may have solutions far away from $t=0$, while the limit cycles are represented only by those solutions that bifurcate from $t=0$. If $n=3$ for example, Mourtada~\cite{Mourtada3} proved that $e(3)=3$, while Panazzolo~\cite{Panazzolo} proved that $\operatorname{fp}(3)=5$.

We observe that Propositions~\ref{Main2} and~\ref{Main3} can be used to prove that a given equation of the form \eqref{47} may be realizable by a family of polynomial vector fields of degree $n$. More precisely, Proposition~\ref{Main2} ensures that any prescribed set of hyperbolicity ratios $(r_1(0),\dots,r_n(0))$ is realizable, while Proposition~\ref{Main3} provides a perturbation family $X_\mu$ such that the map $\mu\mapsto(b_1(\mu),\dots,b_n(\mu))$ has full rank at $\mu=0$. 

In other words (see Roussarie~\cite[Section~$5.4.2$]{Roussarie}), for any prescribed initial condition we have a \emph{generic unfolding} realizable by a family of polynomial vector fields of degree $n$. 

For more details we refer to Roussarie~\cite[Chapter~$5$]{Roussarie} and Panazzolo~\cite{Panazzolo}. Since the unfolding of the first return map of a hyperbolic polycycle is also intrinsically linked with the unfoldings of the \emph{Dulac map} of its hyperbolic saddles, we also refer to the recent works of Marin and Villadelprat \cites{MarVil2020,MarVil2021,MarVil2024}.

\section*{Acknowledgments}

We thank the reviewers for their careful and thoughtful comments and suggestions which help us to improve the presentation of this paper. This work is supported by the Spanish State Research Agency, through the projects PID2022-136613NB-I00 grant and the Severo Ochoa and Mar\'ia de Maeztu Program for Centers and Units of Excellence in R\&D (CEX2020-001084-M),  grant 2021-SGR-00113 from AGAUR, Generalitat de Ca\-ta\-lu\-nya, by CNPq, grant 304798/2019-3, by Agence Nationale de la Recherche (ANR), project ANR-23-CE40-0028, and S\~ao Paulo Research Foundation (FAPESP), grants 2019/10269-3, 2021/01799-9, 2022/14353-1 and 2023/02959-5.

\end{document}